\documentclass[onefignum,onetabnum]{siamart251104}

\usepackage{microtype}
\usepackage{mathtools}
\usepackage{amsmath,amssymb}
\usepackage{booktabs}
\usepackage{subcaption}

\usepackage{lipsum}
\usepackage{amsfonts}
\usepackage{graphicx}
\usepackage{epstopdf}
\ifpdf
  \DeclareGraphicsExtensions{.eps,.pdf,.png,.jpg}
\else
  \DeclareGraphicsExtensions{.eps}
\fi

\newsiamremark{remark}{Remark}
\headers{A SGIP Method for Reaction-Diffusion-Advection Equations}{Boyi Hu, Zhongjian Wang, Jack Xin, Zhiwen Zhang}

\title{A Stochastic Genetic Interacting Particle Method for Reaction-Diffusion-Advection Equations\thanks{To appear in SIAM Journal on Scientific Computing.}}

\author{Boyi Hu\thanks{Department of Mathematics, The University of Hong Kong, Pokfulam Road, Hong Kong SAR, P.R.China (\email{huby22@connect.hku.hk}).}
\and 
Zhongjian Wang\thanks{Division of Mathematical Sciences, School of Physical and Mathematical Sciences, Nanyang Technological University, 21 Nanyang Link, Singapore 637371 (\email{zhongjian.wang@ntu.edu.sg}).}
\and
Jack Xin\thanks{Department of Mathematics, University of California at Irvine, Irvine, CA 92697, USA (\email{jack.xin@uci.edu}).}
\and
Zhiwen Zhang\thanks{Corresponding author. Department of Mathematics, The University of Hong Kong, Pokfulam Road, Hong Kong SAR, P.R.China. Materials Innovation Institute for Life Sciences and Energy (MILES), HKU-SIRI, Shenzhen, 518045, P.R. China (\email{zhangzw@hku.hk}).}
}

\usepackage{amsopn}

\usepackage{algorithm}
\usepackage{algorithmicx}
\usepackage{algpseudocode}

\newcommand{\wt}[1]{\widetilde{#1}}

\newcommand{\by}[1]{{\color{black}#1}}

\newtheorem{assumption}{Assumption}

\begin{document}

\maketitle

\begin{abstract}
We develop and analyze a stochastic genetic interacting particle method (SGIP) for reaction-diffusion-advection (RDA) equations. The SGIP method employs operator splitting to approximate the advection-diffusion and reaction processes, treating the former using particle drift-diffusion and the latter via exact or implicit integration of reaction dynamics over bins, where particle density is estimated using a histogram. A key innovation is the incorporation of adaptive resampling to close the loop of particle and density field description of solutions, mimicking the selection mechanism in genetics. Resampling is also crucial for maintaining long-term stability by redistributing particles in accordance with the evolving density field. We provide a comprehensive error analysis and establish convergence bounds under appropriate regularity assumptions. Numerical experiments in one to three space dimensions demonstrate the method's effectiveness across various reaction types (Fisher-Kolmogorov-Petrovsky-Piskunov (FKPP), cubic, Arrhenius) and flow configurations (shear, cellular, cat's eye, Arnold-Beltrami-Childress (ABC) flows), showing excellent agreement with the finite difference method (FDM) while offering computational advantages for complex flow geometries and higher-dimensional problems.
\end{abstract}



\section{Introduction}\label{section_intro}
Reaction-Diffusion-Advection (RDA) equations play a fundamental role in modeling numerous physical, biological, and chemical processes. The general form of RDA equations is given by
\begin{equation}
\frac{\partial u}{\partial t} + \nabla \cdot (\mathbf{v} u) = D\Delta u + r(u)
\end{equation}
where $u(\mathbf{x}, t)$ represents the concentration or density field, $\mathbf{v}(\mathbf{x}, t)$ is the ambient fluid flow field, $D$ is the molecular diffusion coefficient, and $r(u)$ is the reaction nonlinearity. The versatility of RDA equations is reflected in their wide-ranging applications \cite{chorin1973numerical, peskin1977numerical,williams2018combustion}. In ecology, the Fisher-Kolmogorov-Petrovsky-Piskunov (FKPP) equation models the propagation of population and gene waves \cite{fisher1937wave}, whereas developmental biology employs Turing's reaction-diffusion model to explain morphological patterning \cite{turing1952chemical}. \by{RDA equations also extend to materials science for simulating phase separation via the Allen--Cahn equation \cite{allen1979microscopic}}, and to finance, where they underpin the Black-Scholes model for option pricing \cite{black1973pricing}, collectively underscoring their broad utility \cite{hundsdorfer2013numerical}.

The development of efficient numerical methods for these equations remains an active research area, with challenges from nonlinearities, multiscale phenomena, and high-dimensionality. The finite difference method (FDM) is well-suited for simple geometries, with established stability and convergence properties \cite{leveque2007finite, hoff1978stability}. For complex domains, the Finite Element Method (FEM) offers flexibility, grounded in variational principles \cite{komala2018finite, hughes2012finite, shen2013finite}. Operator splitting techniques (e.g., Lie-Trotter and Strang splitting) are often used to mitigate stiffness from diffusion or reaction terms \cite{trotter1959product,strang1968construction}. Despite their successes, mesh-based methods can become computationally expensive in 3D when advection produces thin and strongly deformed fronts that require fine spatial resolution or adaptive mesh refinement.

%In addition to mesh-based methods, 
Particle (Lagrangian) techniques can transport strongly deformed fronts without a fine Eulerian grid for advection and diffusion. In the 1980s, 
Sherman and Peskin \cite{sherman1986monte,sherman1988solving} introduced a random walk method on the spatial derivative of solutions. The reaction is treated through a reproduction and destruction process to evolve steep wavefronts, with applications to Hodgkin-Huxley equations. Puckett \cite{puckett1989convergence} proved the convergence of a random gradient method for the FKPP equation, where the solution gradient is represented by an empirical measure of particles with weights corresponding to jumps in a step-function approximation. These early works are limited to the one-dimensional (1D) case and to monotone initial data, however. For recent efforts on particle methods  
of RDA-type equations, we refer to   \cite{lyu2022convergent,ParWasser_2025,
zhang2025convergent,lyu2025stochastic} 
and references therein. In the absence of low-order reactions, the current authors \cite{wang2025novel, hu2024stochastic, hu2025convergence} recently introduced efficient stochastic particle-field methods to compute parabolic-parabolic chemotaxis and haptotaxis advection-diffusion systems up to 3D, instead of a history-dependent pure particle description involving heat kernels. 
%In general, a particle description alone can be either computationally expensive or insufficient to represent solutions of physical and biological systems. 

In the spirit of particle-field representation, the stochastic genetic interacting particle (SGIP) method developed in this work provides an efficient and robust framework for RDA equations in several space dimensions. The method naturally extends the SGIP framework \cite{lyu2022convergent, zhang2025convergent}, originally constructed from the Feynman-Kac semigroup for linear RDA equations, to handle the nonlinear and multi-dimensional settings. The algorithm consists of a 3-step operator-splitting to decompose the RDA process: drift-diffusion, density update, and particle resampling. 

The drift-diffusion process of the particles handles the advection-diffusion part of the PDE system through its probabilistic representation (Algorithm~\ref{alg:transport_reaction}). 
%It is a natural extension of the SGIP  \cite{lyu2022convergent, zhang2025convergent} built from the Feynman-Kac semigroup for linear RDA equations.
%through 
%a Lagrangian method, with computational elements following 
%stochastic trajectories governed by the underlying flow field and diffusion 
% Eq.\eqref{advection_diffusion_particle}. 
The reaction process is approximated by going between particles and density: we reconstruct the density field by the particle locations, and evolve in time via reaction ODEs by either an exact (e.g., for FKPP) or a stable implicit integration (Algorithm~\ref{alg:transport_reaction}). %The cycle is then completed by a novel 
 Then, particles are re-drawn from the updated density using an adaptive resampling step (Algorithm~\ref{alg:resampling}) so that they are ready for advection-diffusion again, closing the loop. The joint particle-density evolution is a way to overcome the difficulty of seeking a self-contained particle description. By comparison, in SGIP for linear RDA equations \cite{lyu2022convergent, zhang2025convergent}, particle dynamics form a closed loop through drift-diffusion and resampling without density-field reconstruction.

\by{The SGIP method provides a simple-to-program, scalable, low-cost, and stable framework for multidimensional RDA problems, with particular emphasis on 2D and 3D settings arising in applications such as combustion, biological pattern formation, and phase-field dynamics \cite{williams2018combustion,turing1952chemical,allen1979microscopic}. Its practical advantage is to reduce the cost of the transport step and make challenging 3D complex-flow simulations feasible where fine-grid Eulerian methods can become prohibitively expensive.}
 Furthermore, the SGIP framework here is supported by a complete convergence analysis under appropriate regularity assumptions, which establishes consistent accuracy controlled by time step, spatial resolution, and the number of particles.
 
Our numerical experiments comprehensively demonstrate SGIP's capabilities across various dimensions and complexities, revealing its superior computational efficiency—particularly in demanding 3D settings where traditional mesh-based methods become prohibitively expensive. In 1D scenarios, SGIP shows excellent agreement with FDM for different reaction types (FKPP, cubic, and Arrhenius) \cite{fife2013mathematical, arrhenius1889reaktionsgeschwindigkeit}, accurately capturing propagation speeds and front shapes. The method effectively resolves a spectrum of complex flow configurations, from fundamental 2D shear, cellular, and cat's eye flows to challenging 3D Arnold-Beltrami-Childress (ABC) flows \cite{ranz1979applications, constantin2000bulk, pierrehumbert1991large, dombre1986chaotic, xin2000front}, even when coupled with Lipschitz-continuous reaction terms.  In these scenarios, particularly in 3D, SGIP maintains high accuracy while offering significant computational advantages over traditional mesh-based methods, such as the FDM. As the diffusion coefficient decreases and advection dominates, the SGIP framework remains stable and efficient, in stark contrast to FDM, which suffers from prohibitive computational costs and severe numerical instability under such conditions.

 %These combined strengths allow SGIP to open new avenues for simulating complex multiscale phenomena in physical and biological systems, particularly in scenarios involving intricate flow geometries and requiring long-time integration.

The rest of the paper is organized as follows. In Section~\ref{section_alg}, we present the detailed formulation of the SGIP algorithm. This section is divided into three subsections dedicated to the stochastic advection-diffusion step, the reaction step, and the adaptive resampling strategy, respectively. Section~\ref{section_analysis} is devoted to the theoretical analysis, where a comprehensive error decomposition is provided, and convergence bounds are established. In Section~\ref{section_num}, extensive numerical experiments are conducted to validate the method's performance, ranging from 1D benchmark tests to 2D and 3D simulations involving complex flows. Finally, we conclude with a summary of our findings and a discussion of future research directions in Section~\ref{section_con}.

\section{SGIP Algorithm}\label{section_alg}

In this section, we will present the SGIP algorithm for solving the reaction-diffusion-advection (RDA) equations. 
The mathematical foundation of our method is the operator-splitting technique. In particular, we consider the RDA equation as an abstract evolution equation:
\begin{equation}\label{split}
\frac{\partial u}{\partial t} = (\mathcal{A} + \mathcal{R})u, \quad u(0) = u_0,
\end{equation}
where $\mathcal{A}u = D\Delta u-\nabla \cdot (\mathbf{v} u)$ is the advection-diffusion operator and $\mathcal{R}u = r(u)$ is the reaction operator. Following the Lie-Trotter splitting scheme, we approximate the solution evolution by decoupling the operators:
\begin{equation}\label{lie_split}
u_{n+1} \approx \mathcal{S}_{\mathcal{R}}(\Delta t) \circ \mathcal{S}_{\mathcal{A}}(\Delta t)(u_n).  
\end{equation}
Here, $u_n$ denotes the numerical approximation of $u(\mathbf{x}, t_n)$ at time $t_n$, and $u_{n+1}$ denotes the approximation at the next time level $t_{n+1} = t_n + \Delta t$. The operators \(\mathcal{S}_{\mathcal{A}}(\Delta t)\) and \(\mathcal{S}_{\mathcal{R}}(\Delta t)\) are numerical approximations over a single time step \(\Delta t\) to the exact solutions of the subproblems $\partial_t u = \mathcal{A}u$ and $\partial_t u = \mathcal{R}u$,
respectively. This splitting enables us to employ specialized numerical techniques for each physical process, ensuring its long-term stability and efficiency. 

For the advection-diffusion part $\mathcal{S}_{\mathcal{A}}(\Delta t)$, we employ the stochastic particle method, which turns out to be stable in various high-dimensional advection-dominant problems. For the reaction counterpart $\mathcal{S}_{\mathcal{R}}({\Delta t})$, we apply a density-based approach, which recovers the density $u$ from the particle formulation and approximates the reaction term by resampling. Details of each part of the algorithm are listed in the following subsections.

\subsection{Initialization}
Let $\Omega = [-L, L]^d$ be the computational domain. We approximate the density field $u(\mathbf{x},t)$ using an empirical measure of $N$ particles:
\begin{equation}
u(\mathbf{x},t) \approx \frac{M_0}{N} \sum_{i=1}^N \delta(\mathbf{x} - \mathbf{X}^i(t)),
\end{equation}
where $M_0 = \int_\Omega u_0(\mathbf{x}) d\mathbf{x}$ is the initial total mass and $\{\mathbf{X}^i(t)\}_{i=1}^N$ are particle positions. 
The initial particle positions at $t_0 = 0$ are sampled from the normalized initial distribution:
\begin{equation}
\mathbf{X}_0^i \sim \rho_0(\mathbf{x}) = \frac{u_0(\mathbf{x})}{M_0}, \quad i = 1,\dots,N.
\end{equation}
Particle masses are initialized uniformly according to
\begin{equation}
m_0^i = \frac{M_0}{N}, \quad i = 1,\dots,N.
\end{equation}

In the temporal direction, we assume that the computational time range $[0, T]$ is partitioned into time steps defined by $\{t_n\}_{n=0}^{n_T}$, where $t_0 = 0$ and $t_{n_T} = T$. We consider the uniform time step, i.e., there exists $\Delta t = t_{n+1} - t_n$ for $n = 0, \dots, n_T-1$. For ease of presenting our algorithm, we will use a slight abuse of notation. We will represent the density $u$ at time $t_n$ as $u_n$, and the particle positions at time $t_n$ as $\{\mathbf{X}_n^i\}_{i=1}^N$. Similarly, particle masses at time $t_n$ are denoted as $m_n^i$.

\subsection{Advection-Diffusion Step}

The advection-diffusion sub-step $\mathcal{S}_{\mathcal{A}}({\Delta t})$ advances the solution from $t_n$ to an intermediate time. Each particle follows the stochastic differential equation:
\begin{equation}
d\mathbf{X}^i = \mathbf{v}(\mathbf{X}^i, t)dt + \sqrt{2D}d\mathbf{W}^i(t),
\end{equation}
where $\mathbf{v}(\mathbf{x},t)$ is the velocity field and $\mathbf{W}^i(t)$ are independent Brownian motions.

Using the Euler--Maruyama scheme over $[t_n,t_{n+1}]$, we obtain
\begin{equation}\label{advection_diffusion_particle}
\mathbf{X}_*^i = \mathbf{X}_n^i + \mathbf{v}(\mathbf{X}_n^i, t_n)\Delta t + \sqrt{2D\Delta t}\boldsymbol{\xi}_n^i,
\end{equation}
where $\boldsymbol{\xi}_n^i \sim \mathcal{N}(0,\mathbf{I}_d)$ are independent standard normal random variables, and $\mathbf{X}_*^i$ denotes the intermediate particle positions. \by{The advection-diffusion update and the subsequent density-reaction update are summarized together in Algorithm~\ref{alg:transport_reaction}.}

\subsection{Reaction Step}

The reaction sub-step $\mathcal{S}_{\mathcal{R}}(\Delta t)$ handles the nonlinear reaction term from the intermediate state to time $t_{n+1}$. This step consists of two main components: density estimation on an Eulerian grid, followed by reaction integration.

\subsubsection{Density Estimation}

To evaluate the reaction terms, we first reconstruct the density field from the particle distribution onto a uniform Cartesian grid covering the domain $\Omega = [-L, L]^d$. Let $K$ be the number of grid intervals in each spatial direction, giving a cell size of $\Delta x = 2L/K$. We define the index set of all grid cells as
\begin{equation}
\mathcal{G} := \{1, \dots, K\}^d .
\end{equation}
The cell centers, denoted by $\{\mathbf{x}_j\}_{j \in \mathcal{G}}$, serve as the reconstruction points. The density $\{\hat{u}_*^j\}_{j \in \mathcal{G}}$ in each bin is estimated by counting particles and summing their masses:

\begin{equation}
\hat{u}_*^j=\frac{1}{(\Delta x)^d} \sum_{i:\mathbf{X}_*^i \in \text{bin}_j} m_n^i.
\end{equation}
This provides a piecewise constant approximation of the density field at the intermediate state, which serves as the initial condition for the reaction integration.

\subsubsection{Reaction Integration}

Given the density estimates $\{\hat{u}_*^j\}$, we solve the reaction equation $\partial_t u = r(u)$ over the time interval $\Delta t$ independently in each bin. We consider two cases based on the form of the reaction term:

\paragraph{Closed-form Solutions}
For certain reaction terms, exact integration is possible. A prominent example is the FKPP reaction:
$r(u) = u(1 - u)$,
which admits the exact solution:
$\hat{u}_{n+1}^j = \frac{\hat{u}_*^j e^{\Delta t}}{1 + \hat{u}_*^j(e^{\Delta t} - 1)}$.
Other examples include linear reactions $r(u) = \lambda u$, which yield exponential growth or decay $\hat{u}_{n+1}^j = \hat{u}_*^j e^{\lambda\Delta t}$.

\paragraph{General Reaction Terms}
For reaction terms without closed-form solutions:
\begin{itemize}
    \item Cubic reaction: $r(u) = u^2(1 - u)$,
    \item Arrhenius-type reaction: $r(u) = e^{-E/u}(1 - u)$ with activation energy $E > 0$,
\end{itemize}
an implicit time discretization is adopted to maintain numerical stability for stiff reaction dynamics. The backward Euler scheme reads:
\begin{equation}
\frac{\hat{u}_{n+1}^j - \hat{u}_*^j}{\Delta t} = r(\hat{u}_{n+1}^j),
\end{equation}
which is solved numerically using Newton's method. \by{For stiff problems, higher-order 
implicit methods with strong stiff damping properties, such as the second-order backward differentiation 
formula (BDF2) \cite{gear1971numerical} or L-stable implicit Runge-Kutta methods (e.g., Radau IIA \cite{wanner1996solving}), can be employed.}

The reaction integration is performed independently in each spatial bin, making this step highly parallelizable. After reaction integration, we obtain the updated density field $\{\hat{u}_{n+1}^j\}_{j \in \mathcal{G}}$ at time $t_{n+1}$. \by{Algorithm~\ref{alg:transport_reaction} summarizes the particle advection-diffusion step and combines density estimation and reaction integration into a single binwise update.}

\begin{algorithm}[!t]
\caption{\by{RDA Update}}
\label{alg:transport_reaction}
\begin{algorithmic}[1]
\Require \by{Particles $\{\mathbf{X}_n^i\}_{i=1}^N$, masses $\{m_n^i\}$, velocity field $\mathbf{v}$, reaction term $r$, diffusion coefficient $D$, $\Delta t$, bin size $\Delta x$}
\Ensure \by{Intermediate particles $\{\mathbf{X}_*^i\}_{i=1}^N$, updated density $\{\hat{u}_{n+1}^j\}_{j \in \mathcal{G}}$, total mass $M_{n+1}$}
\For{\by{$i=1$ \textbf{to} $N$}} \Comment{\by{Advection-Diffusion Step}}
    \State \by{Sample $\boldsymbol{\xi}_n^i \sim \mathcal{N}(0,\mathbf{I}_d)$}
    \State \by{$\mathbf{X}_*^i \gets \mathbf{X}_n^i + \mathbf{v}(\mathbf{X}_n^i, t_n)\Delta t + \sqrt{2D\Delta t}\boldsymbol{\xi}_n^i$}
\EndFor
\For{\by{each bin $j$}} \Comment{\by{Reaction Step (illustrated for $r(u)=u(1-u)$)}}
    \State \by{$\displaystyle \hat{u}_*^j \gets
    \frac{1}{(\Delta x)^d}\sum_{i:\mathbf{X}_*^i\in\mathrm{bin}_j}m_n^i,
    \qquad \hat{u}_{n+1}^j \gets
    \frac{e^{\Delta t}\hat{u}_*^j}{1+(e^{\Delta t}-1)\hat{u}_*^j}$}
\EndFor
\State \by{$M_{n+1} \gets \sum_{j \in \mathcal{G}} \hat{u}_{n+1}^j (\Delta x)^d$}
\end{algorithmic}
\end{algorithm}

\subsection{Resampling Step}

Following the reaction integration that produces the updated density field $\{\hat{u}_{n+1}^j\}_{j \in \mathcal{G}}$, we employ a resampling strategy to redistribute particles according to this new density distribution. This step redistributes particles while preserving the total mass produced by the reaction update.

The probability for a particle to be resampled in bin $j$ is proportional to the mass contained in that bin:

\begin{equation}
p_j = \frac{\hat{u}_{n+1}^j \cdot (\Delta x)^d}{M_{n+1}}, \quad j \in \mathcal{G},
\end{equation}
where $M_{n+1} = \sum_{j \in \mathcal{G}} \hat{u}_{n+1}^j \cdot (\Delta x)^d$ is the total mass after reaction computed in Algorithm~\ref{alg:transport_reaction}. This probabilistic framework introduces a selection pressure analogous to natural selection in genetics, in that particles are more likely to survive in a bin—and thus be propagated to the next generation—if that bin contains more mass (is ``healthier'') than its neighbors. Beyond this biological analogy, resampling is critical for long-term numerical stability and computational efficiency, as it continuously rebalances the particle ensemble to control statistical variance.

To implement this redistribution, the target particle counts per bin are drawn from a multinomial distribution parameterized by these probabilities: 
\begin{equation}
(n_j)_{j \in \mathcal{G}} \sim \text{Multinomial}(N, (p_j)_{j \in \mathcal{G}}).
\end{equation}
The resampling strategy employs different sampling methods based on the relationship between target and current particle counts:
\begin{itemize}
    \item When $n_j \leq c_j$: Sample without replacement to preserve diversity
    \item When $n_j > c_j$: Sample with replacement to meet the target count
    \item When $c_j = 0$: Generate new particles uniformly within the bin
\end{itemize}

Here, the $\text{RandomChoice}(\mathcal{X}_j, n_j, \text{mode})$ operation selects $n_j$ particles from the current set $\mathcal{X}_j$ based on a discrete uniform distribution, treating all particles within the bin as equiprobable due to the piecewise constant density assumption. The distinction between the two modes is crucial: ``without replacement'' ($n_j \le c_j$) acts as a downsampling step that selects distinct particles to prevent redundancy and maximize diversity, whereas ``with replacement'' ($n_j > c_j$) acts as a reproduction step where particles are drawn independently (effectively cloning selected particles) to satisfy the target count.

Algorithm~\ref{alg:resampling} implements the resampling strategy described above, adaptively redistributing particles according to the updated density field while preserving statistical accuracy. This approach minimizes statistical errors and maintains a good representation of the density field while adapting to the evolving solution structure.

Combining all components developed in the previous subsections, we present the complete SGIP algorithm for RDA equations in Algorithm~\ref{alg:main}.

\begin{algorithm}[!t]
\caption{Resampling Procedure}
\label{alg:resampling}
\begin{algorithmic}[1]
\Require Intermediate particle positions $\{\mathbf{X}_*^i\}_{i=1}^{N}$, updated density field $\{\hat{u}_{n+1}^j\}_{j \in \mathcal{G}}$, total mass $M_{n+1}$
\Ensure Resampled particles $\{\mathbf{X}_{n+1}^i\}_{i=1}^N$ with updated masses $\{m_{n+1}^i\}_{i=1}^N$
\State Compute $p_j = \hat{u}_{n+1}^j \cdot (\Delta x)^d / M_{n+1}$ for $j \in \mathcal{G}$
\State Sample $(n_j)_{j \in \mathcal{G}} \sim \text{Multinomial}(N, (p_j)_{j \in \mathcal{G}})$
\State Initialize empty list $\mathbf{X}_{\text{new}}$
\For{$j \in \mathcal{G}$}
    \State $c_j \gets \left|\{ \mathbf{X}_*^i \in \text{bin}_j \}\right|$ \Comment{Current particle count in bin $j$}
    \If{$c_j > 0$}
        \State $\mathcal{X}_j \gets \{ \mathbf{X}_*^i \mid \mathbf{X}_*^i \in \text{bin}_j \}$ \Comment{Get particles in bin $j$}
        \If{$n_j \leq c_j$}
            \State $\mathbf{X}_{\text{new},j} \gets \text{RandomChoice}(\mathcal{X}_j, n_j, \text{without replacement})$
        \Else
            \State $\mathbf{X}_{\text{new},j} \gets \text{RandomChoice}(\mathcal{X}_j, n_j, \text{with replacement})$
        \EndIf
    \Else
        \State $\mathbf{X}_{\text{new},j} \gets \text{UniformSample}(\text{bin}_j, n_j)$ \Comment{Sample uniformly from empty bin}
    \EndIf
    \State Append $\mathbf{X}_{\text{new},j}$ to $\mathbf{X}_{\text{new}}$
\EndFor
\State $\{\mathbf{X}_{n+1}^i\}_{i=1}^N \gets \mathbf{X}_{\text{new}}$ \Comment{Concatenate all new particles}
\State Set $m_{n+1}^i = M_{n+1}/N$ for all $i=1,\dots,N$
\end{algorithmic}
\end{algorithm}

\begin{algorithm}[!t]
\caption{SGIP Method for RDA Equations}
\label{alg:main}
\begin{algorithmic}[1]
\Require $N$, $T$, $\Delta t$, $L$, $d$, $u_0(\mathbf{x})$, $\mathbf{v}(\mathbf{x},t)$, $D$, $r(u)$, $K$
\Ensure Estimated density $\{\hat{u}_n\}_{n=0}^{n_T}$

\State \textbf{Initialization:}
\State $M_0 \gets \int_\Omega u_0(\mathbf{x})\,d\mathbf{x}$
\State $\Delta x \gets 2L/K$
\State Sample $\mathbf{X}_0^i \sim u_0(\mathbf{x})/M_0$ for $i=1,\dots,N$
\State Set $m_0^i = M_0/N$ for $i=1,\dots,N$
\State $\hat{u}_0^j \gets (\Delta x)^{-d}\sum_{i:\mathbf{X}_0^i\in\mathrm{bin}_j}m_0^i$ for each bin $j$

\For{$n=0$ \textbf{to} $n_T-1$}
    \State \by{\textbf{Advection-Diffusion and Reaction Update:}}
    \State \by{$\{\mathbf{X}_*^i\}, \{\hat{u}_{n+1}^j\}, M_{n+1} \gets \text{Algorithm~\ref{alg:transport_reaction}}(\{\mathbf{X}_n^i\}, \{m_n^i\}, \mathbf{v}, r, D, \Delta t, \Delta x)$}
    
    \State \textbf{Resampling Step:}
    \State \by{$\{\mathbf{X}_{n+1}^i\}, \{m_{n+1}^i\} \gets \text{Algorithm~\ref{alg:resampling}}(\{\mathbf{X}_*^i\}, \{\hat{u}_{n+1}^j\}, M_{n+1})$}
\EndFor
\State \textbf{return} $\{\hat{u}_n\}_{n=0}^{n_T}$
\end{algorithmic}
\end{algorithm}

Each time step requires $\mathcal{O}(N)$ operations for particle motion, $\mathcal{O}(N)$ for bin assignment, and $\mathcal{O}(|\mathcal{G}|)$ for density estimation and resampling. The overall complexity per time step is $\mathcal{O}(N + |\mathcal{G}|)$, making the method efficient for moderate-dimensional problems.

\begin{remark}[Potential for adaptive acceleration]\label{remark_alg}
\by{For FKPP-type problems with $0\leq u\leq 1$, a propagating front separates nearly saturated ($u\approx 1$) and nearly uninvaded ($u\approx 0$) regions, while most spatial variation lies near the transition. A future adaptive implementation could track this region, concentrate particles and fine reconstruction bins near the front, and use a coarser representation elsewhere. If only $|\mathcal{G}_{\mathrm{active}}|\ll|\mathcal{G}|$ bins require full density estimation and resampling, the binwise cost could be reduced and the particle budget adjusted accordingly. This may be useful for more demanding problems, especially in 3D, although the resulting savings and accuracy remain to be assessed.}
\end{remark}

\section{Error Analysis of the SGIP method for RDA Equation}\label{section_analysis}

In this section, we present a complete error analysis for the SGIP algorithm applied to RDA equations.
We decompose the SGIP error into consistency error from splitting and spatial discretization, propagation error, and stochastic error from density estimation and resampling. Bounds for these components yield the global estimate.

We recall the operator-splitting framework for the RDA equation introduced in Eqs.~\eqref{split}--\eqref{lie_split}:
\begin{equation}
\frac{\partial u}{\partial t} = \mathcal{A}u + \mathcal{R}u, \quad \text{where} \quad \mathcal{A}u = D\Delta u - \nabla \cdot (\mathbf{v} u), \quad \mathcal{R}u = r(u).
\end{equation}
Let $\mathcal{S}_{\mathcal{A}}(t)$ and $\mathcal{S}_{\mathcal{R}}(t)$ denote the exact solution semigroups generated by $\mathcal{A}$ and $\mathcal{R}$, respectively. The ideal Lie-Trotter splitting evolution over a time step $\Delta t$ is given by $\mathcal{S}_{\mathcal{R}}(\Delta t) \circ \mathcal{S}_{\mathcal{A}}(\Delta t)$.

Let $u(t)$ denote the exact solution of the RDA equation, and $\wt{u}_n$ denote the numerical SGIP solution at time $t_n = n\Delta t$. To establish precise error bounds, we introduce the following technical assumption:

\begin{assumption}[Regularity and boundedness]\label{ass:regularity}
The solution and operators satisfy:
\begin{enumerate}
    \item[(a)] \textbf{Initial regularity:} The initial datum $u_0$ belongs to the Sobolev space $H^s(\Omega)$ with $s > \frac{d}{2} + 3$, ensuring sufficient spatial regularity for the error analysis.

    \item[(b)] \textbf{Velocity field regularity:} The velocity field $\mathbf{v}(\mathbf{x},t)$ is uniformly bounded and Lipschitz continuous in space. There exist constants $M_1$ and $L_v$ such that:
    \begin{equation}
    \|\mathbf{v}\|_{L^\infty(\Omega \times [0,T])} \leq M_1, \quad |\mathbf{v}(\mathbf{x},t) - \mathbf{v}(\mathbf{y},t)| \leq L_v |\mathbf{x} - \mathbf{y}|.
    \end{equation}

    \item[(c)] \textbf{Boundedness:} There exists a constant $C_u > 0$ such that the exact solution $u(t,\cdot)$ remains uniformly bounded in the interval $[0, C_u]$ for all $t \in [0,T]$, and maintains smoothness throughout the time evolution. (For the FKPP equation, $C_u = 1$ is the natural carrying capacity.)

    \item[(d)] \textbf{Reaction term regularity:} The reaction function $r(u)$ satisfies the Lipschitz condition
    \begin{equation}
    |r(u) - r(v)| \leq L_r |u - v| \quad \text{for all } u,v \in [0,C_u],
    \end{equation}
    and is twice continuously differentiable with bounded second derivative $|r''(u)| \leq M_2$, ensuring well-posedness of the reaction steps.

    \item[(e)] \textbf{Structural bound preservation:} The reaction function $r(u)$ satisfies
    \begin{equation}
    r(0) \geq 0 \quad \text{and} \quad r(u) \leq 0 \quad \text{for all } u \geq C_u.
    \end{equation}
\end{enumerate}
\end{assumption}

We first present the main convergence theorem, with its proof relying on the technical lemmas established later in this section.

\begin{theorem}[Global error bound]\label{thm:global}
Under Assumption \ref{ass:regularity}, for $n\Delta t \leq T$, the global error of the SGIP algorithm satisfies:
\begin{equation}\label{eq_thm}
\sup_{0\le n\le T/\Delta t} \mathbb{E}\,\|\wt{u}_{n} - u_n\|_{L^2}
\leq C_T \left( \Delta t + \frac{1}{\sqrt{N (\Delta x)^d\Delta t}} + \Delta x \right),
\end{equation}
where the constant $C_T$ depends on the final time $T$, the diffusion parameter $D$, the domain size $L$, the dimension $d$ and the solution regularity.
\end{theorem}

A direct consequence of Theorem~\ref{thm:global} is that, under Assumption~\ref{ass:regularity}, the SGIP solution $\wt{u}_n$ converges to the exact solution $u_n$ in $L^2$-norm as $\Delta t,\Delta x\to0$ and $N\to\infty$, provided that the scaling condition $N(\Delta x)^d\Delta t\to\infty$ is satisfied to ensure the vanishing of statistical variance. \by{The term $1/\sqrt{N(\Delta x)^d\Delta t}$ shows that bin-based density reconstruction remains subject to the curse of dimensionality. Our computational focus is therefore on physically relevant 2D and 3D RDA problems \cite{williams2018combustion,turing1952chemical,allen1979microscopic}, where the Lagrangian treatment can reduce the resolution cost of advection relative to fine-grid Eulerian methods.}

\begin{remark}[Practical convergence and scaling]\label{remark:practical_convergence}
In practical simulations where $\Delta t$ and $\Delta x$ are typically fixed, the error bound \eqref{eq_thm} indicates that the statistical noise is the only term dependent on $N$. Crucially, as demonstrated by the numerical experiments in Section~\ref{section_num}, high accuracy is already attained with moderate values of $\Delta x$, $\Delta t$, and $N$, which is far less restrictive than those suggested by the theoretical analysis. This mild dependence on discretization parameters and sample size enables both robustness and efficiency, even in challenging 3D settings.
\end{remark}

Before presenting the detailed error estimates, we establish the boundedness of the expected total mass of the SGIP solution.

\begin{lemma}[Boundedness of Total Mass]\label{lem:boundedness}
Under Assumption \ref{ass:regularity}, the expected total mass of the SGIP solution remains bounded over the finite time interval $[0, T]$. Specifically, there exists a constant $C_M > 0$ such that for all $n\Delta t \leq T$:
\begin{equation}
\mathbb{E}[M_n] := \mathbb{E}\left[ \int_{\Omega} \wt{u}_n(\mathbf{x}) d\mathbf{x} \right] \leq C_M.
\end{equation}
\end{lemma}

\begin{proof}
The proof relies on the mass evolution properties of the algorithm.
Let $M_n$ be the total mass at step $n$.
Since the advection-diffusion step conserves mass exactly and the resampling step preserves mass in expectation, the evolution of the expected total mass is determined solely by the reaction dynamics.
From Assumption \ref{ass:regularity}(d)(e), we have $r(u) \leq C_u + L_r u$ for $u \ge 0$.
Thus, the growth of total mass is bounded by a linear function of the mass itself:
\begin{equation}
\frac{d}{dt} M(t) = \int_{\Omega} r(u) d\mathbf{x} \leq \int_{\Omega} (C_u + L_r u) d\mathbf{x} = C_u|\Omega| + L_r M(t).
\end{equation}
By Gronwall's inequality, the mass grows at most exponentially. Since we consider a finite time interval $[0, T]$, the expected mass $\mathbb{E}[M_n]$ remains uniformly bounded by a constant $C_M$ dependent on $T, \Omega, L_r$ and the initial mass $M_0$.
\end{proof}

We now analyze the local error of the SGIP by decomposing it into distinct components corresponding to different stages of the algorithm. To distinguish between the accumulation of numerical errors and the irreducible spatial resolution error, we introduce the projection of the exact solution onto the piecewise constant finite volume grid. Let $\bar{u}_n = \mathcal{P}_h u(t_n)$ denote the cell averages of the exact solution at time $t_n$. Unless otherwise stated, all subsequent norms $\|\cdot\|$ will refer to the $L^2$-norms.

We decompose the total error into dynamic and static parts:
\begin{equation}\label{u_total_error}
\|\tilde{u}_{n} - u_{n}\| \leq \underbrace{\|\tilde{u}_{n} - \bar{u}_{n}\|}_{\varepsilon_{\text{dyn}}} + \underbrace{\|\bar{u}_{n} - u_{n}\|}_{\varepsilon_{\text{proj}}}.
\end{equation}
Here, $\varepsilon_{\text{proj}}$ represents the projection error, which is bounded by $O(\Delta x)$ for smooth solutions and does not accumulate over time. We focus our analysis on the evolution of the dynamic error $\varepsilon_{\text{dyn}}$. To define the numerical error precisely, we introduce the ideal discrete evolution operator $\Phi_h(u) = S_{\mathcal{R}}(\Delta t)\mathcal{P}_h S_{\mathcal{A}}(\Delta t)u$, which represents the expected one-step evolution of the algorithm excluding stochastic effects. At time $t_{n+1}$, the error can be decomposed as:
\begin{align}\label{u_split_error}
\begin{aligned}
\|\tilde{u}_{n+1} \!-\! \bar{u}_{n+1}\|\!&\leq \!\underbrace{\|\tilde{u}_{n+1}\!-\!\Phi_h(\tilde{u}_n)\|}_{\varepsilon_{\text{num}}}\!+\! \underbrace{\|\Phi_h(\tilde{u}_n)\!-\!\Phi_h(\bar{u}_n)\|}_{\varepsilon_{\text{propag}}} \!+\! \underbrace{\|\Phi_h(\bar{u}_n)\!-\!\bar{u}_{n+1}\|}_{\varepsilon_{\text{consist}}}.
\end{aligned}
\end{align}
In this decomposition, $\varepsilon_{\mathrm{consist}}$ denotes the local consistency error, which measures the discrepancy between the ideal discrete scheme and the projection of the exact evolution. $\varepsilon_{\mathrm{propag}}$ represents the propagation error, reflecting the stability of the discrete operator acting on the error from the previous step. The term $\varepsilon_{\mathrm{num}}$ captures the stochastic numerical error, arising solely from the finite particle fluctuations around the ideal discrete operator.

The numerical implementation error $\varepsilon_{\mathrm{num}}$ can be further decomposed by tracing through the algorithmic stages. Let $\tilde{u}^{\mathrm{D}}_{n+1}$ denote the numerical density after particle advection-diffusion and binning, $\tilde{u}^{\mathrm{R}}_{n+1}$ the density after reaction integration, and $\tilde{u}_{n+1}$ the final numerical solution after resampling. Consistent with the definition of $\Phi_h$, the reference state for the density estimation is the projected advection $\mathcal{P}_h S_{\mathcal{A}}(\Delta t)\tilde{u}_n$. Then:
\begin{align}\label{u_numerical_error}
\varepsilon_{\mathrm{num}} 
&\leq \underbrace{\|S_{\mathcal{R}}(\Delta t)(\tilde{u}^{\mathrm{D}}_{n+1} - \mathcal{P}_h S_{\mathcal{A}}(\Delta t)\tilde{u}_n)\|}_{\varepsilon_{\mathrm{dens}}} \!+\! \underbrace{\|\tilde{u}^{\mathrm{R}}_{n+1} - S_{\mathcal{R}}(\Delta t)(\tilde{u}^{\mathrm{D}}_{n+1})\|}_{\varepsilon_{\mathrm{react}}} \!+\! \underbrace{\|\tilde{u}_{n+1} - \tilde{u}^{\mathrm{R}}_{n+1}\|}_{\varepsilon_{\mathrm{resamp}}}.
\end{align}
The three terms correspond to distinct error sources: $\varepsilon_{\mathrm{dens}}$ represents the statistical error of density estimation; $\varepsilon_{\mathrm{react}}$ corresponds to the ODE integration error; and $\varepsilon_{\mathrm{resamp}}$ captures the resampling noise introduced during particle redistribution.

The proof of Theorem \ref{thm:global} follows from the subsequent lemmas that bound each error component in Eqs.~\eqref{u_total_error}, \eqref{u_split_error}, and \eqref{u_numerical_error}.

\begin{lemma}[Propagation error estimate]\label{lempropagation_error}
Under Assumption \ref{ass:regularity}, the propagation error of the ideal discrete operator $\Phi_h$ satisfies:
\begin{equation}\label{eq:improved_propagation_error}
\|\Phi_h(\tilde{u}_n) - \Phi_h(\bar{u}_n)\| \leq (1 + C_1\Delta t) \|\tilde{u}_n - \bar{u}_n\|,
\end{equation}
where $C_1$ depends on the Lipschitz constants of the advection-diffusion and reaction operators.
\end{lemma}

\begin{proof}
Let $w = \tilde{u}_n - \bar{u}_n$. We trace the evolution of the norm difference through the operators.

First, apply advection-diffusion: $w_1 = S_{\mathcal{A}}(\Delta t)\tilde{u}_n - S_{\mathcal{A}}(\Delta t)\bar{u}_n$. Standard energy estimates for the advection-diffusion equation $\partial_t w_1 = -\nabla \cdot (\mathbf{v}w_1) + D\Delta w_1$ yield:
\begin{align}
\frac{1}{2}\frac{d}{dt}\|w_1\|^2 &= \langle w_1, -\nabla\cdot(\mathbf{v}w_1)\rangle + D\langle w_1, \Delta w_1\rangle \leq C \|\nabla \cdot \mathbf{v}\|_\infty \|w_1\|^2
\end{align}
Using Assumption \ref{ass:regularity}(b) and Gronwall's inequality, we obtain:
\begin{equation}
\|w_1\| \leq e^{K_A \Delta t} \|w\|,
\end{equation}
where $K_A$ depends on $M_1$ and $L_v$ defined in Assumption \ref{ass:regularity}(b).

Next, the operator $\mathcal{P}_h$ represents the $L^2$-orthogonal projection onto the finite volume space (piecewise constant functions). By the Cauchy-Schwarz inequality, for a mesh cell $B_j$ with volume $|B_j|$, we have:
\begin{equation}
    \int_{B_j} |\mathcal{P}_h w_1|^2 \, d\mathbf{x} 
    = |B_j| \left| \frac{1}{|B_j|} \int_{B_j} w_1(\mathbf{x}) \, d\mathbf{x} \right|^2 
    \leq \int_{B_j} |w_1(\mathbf{x})|^2 \, d\mathbf{x}.
\end{equation}
Summing over all cells $j$ yields the contraction property:
\begin{equation}
    \|\mathcal{P}_h w_1\|^2 = \sum_j \int_{B_j} |\mathcal{P}_h w_1|^2 \, d\mathbf{x} \leq \sum_j \int_{B_j} |w_1|^2 \, d\mathbf{x} = \|w_1\|^2.
\end{equation}

For the reaction step, let $w_2 = \Phi_h(\tilde{u}_n) - \Phi_h(\bar{u}_n)$. By using the Lipschitz property of $r(u)$ with constant $L_r$ (Assumption \ref{ass:regularity}(d)) and Gronwall's inequality, we have:
\begin{align}
\|w_2\| \leq e^{L_r \Delta t} \|\mathcal{P}_h w_1\|.
\end{align}

Combining the estimates from all three steps, we obtain for sufficiently small $\Delta t$ that
\begin{align}
\|\Phi_h(\tilde{u}_n) - \Phi_h(\bar{u}_n)\| 
&\leq e^{(K_A + L_r)\Delta t} \|\tilde{u}_n - \bar{u}_n\| \nonumber \\
&\leq (1 + C_1\Delta t) \|\tilde{u}_n - \bar{u}_n\|,
\end{align}
where $C_1$ is a constant depending on the Lipschitz constants of the advection-diffusion and reaction operators. Thus, the proof is complete.
\end{proof}

\begin{lemma}[Consistency error estimate]\label{lem:splitting_error}
Under Assumption \ref{ass:regularity}, the local consistency error between the ideal discrete scheme and the projected exact evolution satisfies:
\begin{equation}\label{splitting_error}
\varepsilon_{\mathrm{consist}} = \|\Phi_h(\bar{u}_n) - \bar{u}_{n+1}\| \leq C_2(\Delta t)^2 + C_3 \Delta t \Delta x,
\end{equation}
where $C_2$ and $C_3$ depend on the solution regularity and operator bounds.
\end{lemma}

\begin{proof}
Recall $\Phi_h = S_{\mathcal{R}} \mathcal{P}_h S_{\mathcal{A}}$ and $\bar{u}_{n+1} = \mathcal{P}_h S_{\mathcal{R}+\mathcal{A}} u_n$. To facilitate the analysis, we introduce the intermediate reference state $\bar{u}_{n+1}^* = \mathcal{P}_h S_{\mathcal{R}+\mathcal{A}}(\Delta t) \bar{u}_n$, which represents the projected exact evolution starting from $\bar{u}_n$. We decompose the total consistency error into three components:
\begin{align}\label{eq:consistency_decomp}
\varepsilon_{\mathrm{consist}} &\leq \| \Phi_h(\bar{u}_n) - \mathcal{P}_h S_{\mathcal{R}}S_{\mathcal{A}}(\Delta t) \bar{u}_n \| + \| \mathcal{P}_h S_{\mathcal{R}}S_{\mathcal{A}}(\Delta t) \bar{u}_n - \bar{u}_{n+1}^* \| + \| \bar{u}_{n+1}^* - \bar{u}_{n+1} \|.
\end{align}
\textbf{(1) Time Splitting Error.}
The second term represents the classical Lie-Trotter splitting error. From Lemma \ref{lempropagation_error}, we have:
\begin{equation}
\| \mathcal{P}_h S_{\mathcal{R}}S_{\mathcal{A}}(\Delta t) \bar{u}_n - \mathcal{P}_h S_{\mathcal{R}+\mathcal{A}}(\Delta t) \bar{u}_n \| \leq \|S_{\mathcal{R}}S_{\mathcal{A}}(\Delta t)\bar{u}_n - S_{\mathcal{R}+\mathcal{A}}(\Delta t)\bar{u}_n\|
\end{equation}
By the Baker-Campbell-Hausdorff formula, this error is bounded by the commutator of the operators, which is defined as:
\begin{equation}
    [\mathcal{R},\mathcal{A}]u = r'(u)\mathcal{A}u - \mathcal{A}(r(u))
\end{equation}
Expanding the term $\mathcal{A}(r(u))$ using the chain rule reveals that the non-vanishing commutativity arises primarily from spatial derivatives of the nonlinearity, introducing terms proportional to $r''(u)|\nabla u|^2$. Invoking Assumption \ref{ass:regularity}(b) for velocity boundedness, Assumption \ref{ass:regularity}(d) for the bound $|r''(u)| \leq M_2$, and the regularity of the solution, we obtain:
\begin{equation}\label{split_error}
\|S_{\mathcal{R}}S_{\mathcal{A}}\bar{u}_n - S_{\mathcal{R}+\mathcal{A}}\bar{u}_n\| \leq \frac{1}{2}(\Delta t)^2 \|[\mathcal{R}, \mathcal{A}]\bar{u}_n\| \leq C_{\text{split}} (\Delta t)^2.
\end{equation}
Consequently, the splitting error is bounded by $C_{\text{split}}(\Delta t)^2$, where $C_{\text{split}}$ depends on $D$, the operator norms, and solution regularity. 

% The smoothing effect of $S_{\mathcal{A}}$ ensures the validity of the commutator bound despite the piecewise constant nature of $\bar{u}_n$.

\textbf{(2) Spatial Consistency Error.}
The first term in Eq.\eqref{eq:consistency_decomp} quantifies the error arising from the discretization of the reaction step on the grid. Specifically, it measures the discrepancy between applying the reaction operator on the projected density versus projecting the reacted density.

Let $v(\mathbf{x}) = S_{\mathcal{A}}(\Delta t)\bar{u}_n(\mathbf{x})$ denote the density field after the exact advection-diffusion step but before projection. The spatial consistency error is exactly:
\begin{align}
\| \Phi_h(\bar{u}_n) - \mathcal{P}_h S_{\mathcal{R}}S_{\mathcal{A}}(\Delta t) \bar{u}_n \| &= \| S_{\mathcal{R}}(\Delta t)(\mathcal{P}_h v) - \mathcal{P}_h (S_{\mathcal{R}}(\Delta t) v) \|.
\end{align}
Since the reaction function $r(u)$ is twice continuously differentiable (Assumption \ref{ass:regularity}(d)), we can expand the exact reaction evolution $S_{\mathcal{R}}(\Delta t)u$ via Taylor series with an integral remainder. For any state $w$, we have:
\begin{equation}
    S_{\mathcal{R}}(\Delta t)w = w + \Delta t r(w) + \mathcal{E}_t(w),
\end{equation}
where $\|\mathcal{E}_t(w)\| \leq K_r (\Delta t)^2$ and the constant $K_r$ depends on the bounds of $r$ and its derivatives.

Applying this expansion to both terms in the error expression:
\begin{align}
S_{\mathcal{R}}(\Delta t)(\mathcal{P}_h v) &= \mathcal{P}_h v + \Delta t r(\mathcal{P}_h v) + \mathcal{E}_t(\mathcal{P}_h v), \\
\mathcal{P}_h (S_{\mathcal{R}}(\Delta t) v) &= \mathcal{P}_h \left( v + \Delta t r(v) + \mathcal{E}_t(v) \right) = \mathcal{P}_h v + \Delta t \mathcal{P}_h r(v) + \mathcal{P}_h \mathcal{E}_t(v).
\end{align}
Subtracting the two equations and using the triangle inequality yields:
\begin{align}
\| \!S_{\mathcal{R}}(\Delta t)(\mathcal{P}_hv) \!-\! \mathcal{P}_h (S_{\mathcal{R}}(\Delta t) v) \|\!\leq\! \Delta t \| r(\mathcal{P}_h v) \!-\! \mathcal{P}_h r(v) \| \!+\! \| \mathcal{E}_t(\mathcal{P}_h v) \| \!+\! \| \mathcal{P}_h \mathcal{E}_t(v) \|.
\end{align}
The first term $\| r(\mathcal{P}_h v) - \mathcal{P}_h r(v) \|$ represents the commutator between the nonlinear reaction function and the projection operator. Using the Lipschitz property of $r(u)$ with constant $L_r$ and the approximation properties of the projection $\mathcal{P}_h$, this term is bounded by
\begin{equation}
    \| r(\mathcal{P}_h v) - \mathcal{P}_h r(v) \| \leq C_P \Delta x \|\nabla r(v)\| \leq C_P L_r \Delta x \|\nabla v\|.
\end{equation}
Crucially, $v = S_{\mathcal{A}}(\Delta t)\bar{u}_n$ is the result of the advection-diffusion operator acting on the piecewise constant function $\bar{u}_n$. Due to the smoothing property of the parabolic operator for $t > 0$, $v$ belongs to $H^1(\Omega)$. Moreover, given the smoothness of the exact solution (Assumption \ref{ass:regularity}), $\|\nabla v\|$ remains uniformly bounded.

%, ensuring the spatial error estimate is well-defined despite the initial discontinuity of $\bar{u}_n$.

For the remainder terms, since the projection operator is non-expansive in $L^2$, we have:
\begin{equation}
    \| \mathcal{E}_t(\mathcal{P}_h v) \| + \| \mathcal{P}_h \mathcal{E}_t(v) \| \leq 2 K_r (\Delta t)^2.
\end{equation}
Combining these bounds, we obtain:
\begin{equation}\label{spatial_error}
\| \Phi_h(\bar{u}_n) - \mathcal{P}_h S_{\mathcal{R}}S_{\mathcal{A}}(\Delta t) \bar{u}_n \| \leq C_{\text{spatial}} \Delta t \Delta x + 2K_r (\Delta t)^2,
\end{equation}
where $C_{\text{spatial}} = C_P L_r \|\nabla u\|_{L^\infty(0,T; L^2)}$.

\textbf{(3) Projection Stability Error.}
The last term accounts for the evolution of the initial projection error over one time step. Using the first-order Taylor expansion of the exact evolution operator $S_{\mathcal{R}+\mathcal{A}}(\Delta t) = I + \Delta t (\mathcal{R}+\mathcal{A}) + O(\Delta t^2)$, we have:
\begin{align}\label{pro_error}
\|\bar{u}_{n+1}^* - \bar{u}_{n+1}\| &= \|\mathcal{P}_h ((I + \Delta t \mathcal{L})\bar{u}_n - (I + \Delta t \mathcal{L})u_n) + O(\Delta t^2)\| \nonumber \\
&\leq \|\mathcal{P}_h(\bar{u}_n - u_n)\| + \Delta t \|\mathcal{P}_h \mathcal{L}(\bar{u}_n - u_n)\| + C_{\text{rem}}(\Delta t)^2.
\end{align}
Here, we define $\mathcal{L}=\mathcal{R}+\mathcal{A}$. Crucially, since $\bar{u}_n = \mathcal{P}_h u_n$ is the projection of $u_n$, we have $\mathcal{P}_h(\bar{u}_n - u_n) = \mathcal{P}_h(\mathcal{P}_h u_n - u_n) = 0$.
The remaining term is bounded by the Lipschitz continuity of the generator $\mathcal{L}$ and the projection error bound $\|\bar{u}_n - u_n\| \leq C \Delta x$:
\begin{equation}
\Delta t \|\mathcal{P}_h \mathcal{L}(\bar{u}_n - u_n)\| \leq C_{\mathcal{L}} \Delta t \|\bar{u}_n - u_n\| \leq C_{\text{stab}} \Delta t \Delta x.
\end{equation}

Combining the estimates from Eqs.~\eqref{split_error}, \eqref{spatial_error}, and \eqref{pro_error}, we obtain:
\begin{align}
\varepsilon_{\mathrm{consist}} &\leq  (C_{\text{split}} + 2K_r + C_{\text{rem}})(\Delta t)^2 + (C_{\text{spatial}} + C_{\text{stab}}) \Delta t \Delta x.
\end{align}
By defining constants $C_2=C_{\text{split}} + 2K_r + C_{\text{rem}}$ and $C_3=C_{\text{spatial}} + C_{\text{stab}}$, we obtain the stated bound in Eq.\eqref{splitting_error}.
\end{proof}

\begin{lemma}[Statistical density estimation error]\label{lem:density_error}
Under Assumption \ref{ass:regularity}, the density estimation step is unbiased with respect to the ideal discrete state, and its mean squared error satisfies:
\begin{equation}\label{density_error}
\mathbb{E}\left[\|\varepsilon_{\mathrm{dens}}\|^2\right] = \mathbb{E}\!\left[ \big\| S_{\mathcal{R}}(\Delta t)(\wt{u}^{\mathrm{D}}_{n+1}) 
- S_{\mathcal{R}}(\Delta t)(\mathcal{P}_h S_{\mathcal{A}}(\Delta t)\wt{u}_n) \big\|^2 \right] \leq \frac{C_4}{N (\Delta x)^d},
\end{equation}
where $C_4$ depends on the domain size and the uniform bound of the solution.
\end{lemma}

\begin{proof}
Let $v(\mathbf{x}) = S_{\mathcal{A}}(\Delta t)\wt{u}_n(\mathbf{x})$ denote the advected-diffused density field. The reference state is its projection $\bar{v}_j = \frac{1}{(\Delta x)^d} \int_{\text{bin}_j} v(\mathbf{x}) d\mathbf{x}$. The numerical density in bin $j$, $\hat{u}_j$, is constructed from $N$ independent particle positions $\{\mathbf{X}_i\}_{i=1}^N$ sampled from the distribution $v(\mathbf{x})/M$, where $M$ is the total mass.

Let $I_j$ be the number of particles in bin $j$. The variable $I_j$ follows a binomial distribution $I_j \sim \text{Binomial}(N, p_j)$, where the probability $p_j$ is given by:
\begin{equation}
p_j = \int_{\text{bin}_j} \frac{v(\mathbf{x})}{M} d\mathbf{x} = \frac{(\Delta x)^d \bar{v}_j}{M}.
\end{equation}

The estimator is defined as $\hat{u}_j = \frac{M I_j}{N (\Delta x)^d}$. Taking the expectation:
\begin{equation}
\mathbb{E}[\hat{u}_j] = \frac{M}{N (\Delta x)^d} \mathbb{E}[I_j] = \frac{M}{N (\Delta x)^d} (N p_j) = \frac{M}{(\Delta x)^d} \frac{(\Delta x)^d \bar{v}_j}{M} = \bar{v}_j.
\end{equation}
Thus, $\mathbb{E}[\hat{u}_j] - \bar{v}_j = 0$, confirming the estimator is unbiased with respect to the projected state.

The mean squared error is the sum of variances in each bin:
\begin{equation}
\mathbb{E}[\|\wt{u}^{\mathrm{D}}_{n+1} - \bar{v}\|^2] = \sum_{j \in \mathcal{G}} (\Delta x)^d \mathrm{Var}(\hat{u}_j).
\end{equation}
Using the variance of the binomial distribution $\mathrm{Var}(I_j) = N p_j (1 - p_j)$:
\begin{align}
\mathrm{Var}(\hat{u}_j) = \frac{M^2}{N^2 (\Delta x)^{2d}} N p_j (1 - p_j) \leq \frac{M^2}{N (\Delta x)^{2d}} p_j.
\end{align}
Substituting $p_j = \frac{(\Delta x)^d \bar{v}_j}{M}$ into the sum, we have:
\begin{align}
\mathbb{E}[\|\wt{u}^{\mathrm{D}}_{n+1} - \bar{v}\|^2] \leq \sum_{j \in \mathcal{G}} (\Delta x)^d \left( \frac{M^2}{N (\Delta x)^{2d}} \frac{(\Delta x)^d \bar{v}_j}{M} \right) \leq \frac{M^2}{N (\Delta x)^d}.
\end{align}

Since the reaction term satisfies Assumption~\ref{ass:regularity}(d) with Lipschitz constant $L_r$, the associated solution operator $S_{\mathcal{R}}(\Delta t)$ is Lipschitz continuous on $L^2(\Omega)$ with constant $L_{\mathrm{R}} = e^{L_r \Delta t} \leq e^{L_r T}$. Therefore,
\[
\mathbb{E}\!\left[ \varepsilon_{\mathrm{dens}}^2 \right] 
\leq L_{\mathrm{R}}^2 \, \mathbb{E}\!\left[ \big\| \wt{u}^{\mathrm{D}}_{n+1} - \mathcal{P}_h v \big\|^2 \right]
\leq \frac{C_4}{N (\Delta x)^d},
\]
where $C_4 = e^{2 L_r T} (C_M)^2$ depends on the final time $T$ and the mass bound $C_M$ established in Lemma \ref{lem:boundedness}.
\end{proof}

\begin{lemma}[Reaction integration error]\label{lem:reaction_error}
Under Assumption \ref{ass:regularity}, the reaction integration error satisfies:
\begin{equation}\label{reaction_error}
\varepsilon_{\mathrm{react}} =\|\tilde{u}^{\mathrm{R}}_{n+1} - S_{\mathcal{R}}(\Delta t)(\tilde{u}^{\mathrm{D}}_{n+1})\| \leq C_5 (\Delta t)^{p+1},
\end{equation}
where $p$ is the order of the time integration scheme ($p \to \infty$ for exact integration).
\end{lemma}

\begin{proof}
Let $\Psi_{\Delta t}(\cdot)$ denote the numerical integration operator over one time step $\Delta t$, and $S_{\mathcal{R}}(\Delta t)$ be the exact reaction operator. The local truncation error for a $p$-th order method applied to the ODE $\dot{u} = r(u)$ is given pointwise by:
\begin{equation}
| \Psi_{\Delta t}(u(\mathbf{x})) - S_{\mathcal{R}}(\Delta t)u(\mathbf{x}) | \leq C_r (\Delta t)^{p+1},
\end{equation}
where $C_r$ depends on the $(p+1)$-th derivative of the solution, which is bounded due to the smoothness of $r(u)$ and the boundedness of $u$ (Assumption \ref{ass:regularity} (c)(d)).

The $L^2$-norm of the error is:
\begin{align}
\varepsilon_{\mathrm{react}}^2 &= \int_{\Omega} | \Psi_{\Delta t}(\tilde{u}^D(\mathbf{x})) - S_{\mathcal{R}}(\Delta t)\tilde{u}^D(\mathbf{x}) |^2 d\mathbf{x} \leq  |\Omega| C_r^2 (\Delta t)^{2(p+1)}.
\end{align}
Taking the square root, we obtain:
\begin{equation}
\varepsilon_{\mathrm{react}} \leq C_5 (\Delta t)^{p+1}.
\end{equation}
Here, the constant $C_5 = \sqrt{|\Omega|} C_r$.
\end{proof}

\begin{lemma}[Resampling error]\label{lem:resampling_error}
Under Assumption \ref{ass:regularity}, the resampling step introduces a mean-zero stochastic error with variance satisfying:
\begin{equation}\label{resampling_error}
\mathbb{E}[\|\varepsilon_{\mathrm{resamp}}\|^2] = \mathbb{E}[\|\wt{u}_{n+1} - \wt{u}^{\mathrm{R}}_{n+1}\|^2] \leq \frac{C_6}{N (\Delta x)^d},
\end{equation}
where the constant $C_6$ depends on the domain size and dimension.
\end{lemma}

\begin{proof}
Let $w_j = \tilde{u}^{\mathrm{R}}_{n+1, j}$ be the density value in bin $j$ before resampling. The resampling process draws $N$ particles based on the mass distribution. The probability of a new particle falling into bin $j$ is:
\begin{equation}
q_j = \frac{w_j (\Delta x)^d}{M}, \quad \text{where } M = \sum_{k} w_k (\Delta x)^d.
\end{equation}
Let $K_j$ be the number of particles in bin $j$ after resampling. The vector $(K_1, \dots, K_{|\mathcal{G}|})$ follows a multinomial distribution with parameters $N$ and probabilities $\{q_j\}$.
The resampled density is $\tilde{u}_{n+1, j} = \frac{M K_j}{N (\Delta x)^d}$.

Since $\mathbb{E}[K_j] = N q_j$, we have:
\begin{equation}
\mathbb{E}[\tilde{u}_{n+1, j}] = \frac{M}{N (\Delta x)^d} (N q_j) = \frac{M}{(\Delta x)^d} \frac{w_j (\Delta x)^d}{M} = w_j.
\end{equation}
Thus, $\mathbb{E}[\tilde{u}_{n+1}] = \tilde{u}^{\mathrm{R}}_{n+1}$.

The variance of the count in bin $j$ is $\mathrm{Var}(K_j) = N q_j (1 - q_j)$. The expected squared $L^2$ error is:
\begin{align}
\mathbb{E}[\|\tilde{u}_{n+1} - \tilde{u}^{\mathrm{R}}_{n+1}\|^2] 
= \frac{M^2}{N (\Delta x)^d} \sum_{j \in \mathcal{G}} q_j (1 - q_j) \leq \frac{C_6}{N (\Delta x)^d}.
\end{align}
Here, the constant $C_6 = (C_M)^2$, consistent with the mass bound $C_M$ in Lemma \ref{lem:boundedness}.
\end{proof}

Now, we are ready to prove Theorem \ref{thm:global}.

\textit{Proof of Theorem \ref{thm:global}.}
Recall the decomposition of the global error into the dynamic discrete error and the static projection error:
\begin{equation}
\|\wt{u}_{n} - u_n\| \leq \|\wt{u}_{n} - \bar{u}_n\| + \|\bar{u}_{n} - u_n\|.
\end{equation}
The static projection error is bounded by $\|\bar{u}_{n} - u_n\| \leq C_{7} \Delta x$. We define the mean squared dynamic error as $E_n = \mathbb{E}[\|\wt{u}_{n} - \bar{u}_n\|^2]$.

Using the error decomposition in Eq.\eqref{u_split_error} and introducing the ideal discrete operator $\Phi_h$, the error evolution at step $(n+1)$ is:
\begin{equation}
e_{n+1} = \wt{u}_{n+1} - \bar{u}_{n+1} = (\Phi_h(\wt{u}_n) - \Phi_h(\bar{u}_n)) + \varepsilon_\mathrm{consist} + \varepsilon_{\mathrm{num}}.
\end{equation}
Crucially, from Eq.\eqref{u_numerical_error}, the stochastic numerical error $\varepsilon_{\mathrm{dens}} + \varepsilon_{\mathrm{resamp}}$ (comprising density estimation variance and resampling noise) has zero mean conditional on the state at time $t_n$. Therefore, the cross-term in the squared norm expansion vanishes in expectation:
\begin{align}
E_{n+1} &= \mathbb{E}\|\Phi_h(\wt{u}_n) - \Phi_h(\bar{u}_n) + \varepsilon_{\mathrm{consist}} + \varepsilon_{\mathrm{num}} \|^2 \nonumber \\
&= \mathbb{E}\| \Phi_h(\wt{u}_n) - \Phi_h(\bar{u}_n) + \varepsilon_{\mathrm{consist}} + \varepsilon_{\mathrm{react}} \|^2 + \mathbb{E}\| \varepsilon_{\mathrm{dens}} + \varepsilon_{\mathrm{resamp}} \|^2.
\end{align}

We combine the consistency error from Lemma~\ref{lem:splitting_error} with the reaction error from Lemma~\ref{lem:reaction_error}. At this stage, we assume a practical first-order time integration scheme ($p=1$) is used for the reaction step, which yields $\|\varepsilon_{\mathrm{react}}\| \leq C_5 (\Delta t)^2$. Using the inequality $\|a+b\|^2 \leq (1+\Delta t)\|a\|^2 + (1+\frac{1}{\Delta t})\|b\|^2$, Lemma \ref{lempropagation_error}, and Lemma \ref{lem:splitting_error}, we obtain:
\begin{align}
&\mathbb{E}\| \Phi_h(\wt{u}_n) - \Phi_h(\bar{u}_n) + \varepsilon_{\mathrm{consist}} + \varepsilon_{\mathrm{react}} \|^2 \nonumber \\
&\leq (1+\Delta t) \mathbb{E}\| \Phi_h(\wt{u}_n) - \Phi_h(\bar{u}_n) \|^2 + \left(1+\frac{1}{\Delta t}\right) \|\varepsilon_{\mathrm{consist}}+ \varepsilon_{\mathrm{react}}\|^2 \nonumber \\
&\leq (1+\Delta t) (1+C_1 \Delta t)^2 E_n + \frac{2}{\Delta t} \left( (C_2+C_5)(\Delta t)^2 + C_3 \Delta t \Delta x \right)^2 \nonumber \\
&\leq (1 + C_8 \Delta t) E_n + C_9 \Delta t (\Delta t + \Delta x)^2,
\end{align}
where we have simplified terms for small $\Delta t$.

Combining Lemma \ref{lem:density_error} and Lemma \ref{lem:resampling_error}, the total variance is:
\begin{align}
\mathbb{E}\| \varepsilon_{\mathrm{dens}} + \varepsilon_{\mathrm{resamp}} \|^2 \leq \frac{C_4 + C_6}{N (\Delta x)^d} = \frac{C_{10}}{N (\Delta x)^d}.
\end{align}

Substituting these bounds back into the recurrence for $E_n$:
\begin{align}
E_{n+1} \leq (1 + C_8 \Delta t) E_n + C_9 \Delta t (\Delta t + \Delta x)^2 + \frac{C_{10}}{N (\Delta x)^d}.
\end{align}
Since the initial particles are sampled from $u_0/M_0$, the initial bin-density error is unbiased and satisfies $E_0=O(1/[N(\Delta x)^d])$. The discrete Gronwall inequality then gives, for $n\Delta t\leq T$,
\begin{align}
E_n \leq e^{C_8 T} \left( C_9 (\Delta t + \Delta x)^2 + \frac{C_{10}}{N (\Delta x)^d \Delta t} \right).
\end{align}

Finally, taking the square root to obtain the $L^2$-norm estimate, and adding the static projection error $\|\bar{u}_n - u_n\| \leq C_{7} \Delta x$:
\begin{align}
\mathbb{E}\|\wt{u}_{n} - u_n\| &\leq \sqrt{E_n} + C_{7} \Delta x \leq C_T \left( \Delta t + \frac{1}{\sqrt{N (\Delta x)^d \Delta t}} + \Delta x \right).
\end{align}
Here, $C_T = e^{\frac{1}{2}C_8 T} \max\left\{ \sqrt{C_9}, \sqrt{C_{10}} \right\} + C_7$. This completes the proof of Theorem \ref{thm:global}.

\section{Numerical Experiments}\label{section_num}
This section presents a series of numerical experiments designed to validate the accuracy, robustness, and computational efficiency of the proposed SGIP method. We begin with 1D nonlinear reaction models (Subsection~\ref{subsection:1d}), move to front propagation in steady 2D flows (Subsection~\ref{subsection:2d}), and conclude with 3D ABC flows under different diffusion coefficients and reaction models (Subsection~\ref{subsection:3d}).

All the experiments presented in this work were performed on the HPC2021 system at the University of Hong Kong, which is equipped with 16-core Intel Xeon 6226R processors and an NVIDIA Tesla V100 32 GB SXM2 GPU.

\subsection{1D Numerical Experiments}\label{subsection:1d}
In this section, we present 1D numerical experiments to validate the robustness of the SGIP method for solving reaction-diffusion equations. We consider three different nonlinear reaction terms and compare the results with traditional FDM. 

The experiments in 1D are conducted on the computational domain $\Omega = [-L, L]$ with $L = 60$. \by{This value of $L$ is chosen so that the propagating fronts remain well separated from the artificial boundary over the simulated time interval. The 2D and 3D experiments use $\Omega=[-60,60]^d$; for the $2\pi$-periodic flows considered there, the domain contains approximately 19 flow periods in each coordinate direction.} The initial condition is given by a uniform distribution on the interval $[0,1]$:
\begin{equation}
u_0(x) = \chi_{[0,1]}(x) = \begin{cases} 
1 & \text{for } x \in [0, 1], \\
0 & \text{otherwise},
\end{cases}
\end{equation}
where $\chi_{[0,1]}$ denotes the characteristic function of the interval $[0,1]$. This discontinuous initial profile tests the method's capability to handle sharp gradients and evolving interfaces.

For the reference FDM simulations, we employ a high-resolution discretization with spatial step $\delta x = 10^{-3}$ and temporal step $\delta t = 10^{-3}$ to ensure accurate reference solutions. For the SGIP, we use $N = 10^6$ particles to approximate the density distribution, with the spatial domain discretized into $K = 150$ bins for density estimation, corresponding to a bin width of $\Delta x = 0.8$. The time step is set to $\Delta t = 0.5$.

% First, we consider the classical FKPP equation:
% \begin{equation}
%  u_t = u_{xx} + u(1-u).
% \end{equation}

We begin by considering three distinct nonlinear reaction terms. The classical FKPP equation, \(u_t = u_{xx} + u(1-u)\), serves as the baseline. Figure~\ref{uxx+u(1-u)} shows that the SGIP results exhibit excellent agreement with the FDM reference, and the wave front propagates with a constant shape, confirming the characteristic traveling wave structure.

Next, a cubic nonlinearity, \(u_t = u_{xx} + u^2(1-u)\), is examined. As shown in Figure~\ref{uxx+u2(1-u)}, this modification leads to a visibly slower propagation speed compared to the standard FKPP case. Finally, we test an Arrhenius-type reaction term, \(u_t = u_{xx} + e^{-E/u}(1-u)\) with \(E = 1/2\), which introduces a singularity near \(u = 0\). Despite this numerical challenge, the SGIP method performs satisfactorily. Figure~\ref{uxx+exp(1-u)} shows that the wave front propagates at a notably reduced speed while maintaining a coherent profile, further validating the robustness of the method for complex nonlinear reactions.

\begin{figure}[htbp]
    \centering
    \vspace{-0.2cm}
    \includegraphics[width=0.7\textwidth]{picture/kpp_comparison_all_times.pdf}
        \caption{Comparison of FDM and SGIP solutions for $u_t = u_{xx} + u(1-u)$}
   \label{uxx+u(1-u)} 
\end{figure}

% Figure \ref{uxx+u(1-u)} displays the numerical solutions of the standard FKPP equation at different time instances. It can be observed that the SGIP method results show excellent agreement with FDM. As expected for the FKPP equation, the shape of the wave front remains unchanged during propagation, demonstrating the characteristic travelling wave structure. 

% Next, we consider a reaction-diffusion equation with a cubic nonlinear reaction term:
% \begin{equation}
%  u_t = u_{xx} + u^2(1-u).
% \end{equation}

\begin{figure}[htbp]
    \centering
    \vspace{-0.2cm}
    \includegraphics[width=0.7\textwidth]{picture/kpp_comparison_u2_all_times.pdf}
   \caption{Comparison of FDM and SGIP solutions for $u_t = u_{xx} + u^2(1-u)$}
\label{uxx+u2(1-u)}
\end{figure}

% Notably, in Figure \ref{uxx+u2(1-u)}, the cubic nonlinearity \(u^2(1-u)\) leads to a visibly slower propagation speed compared to the standard FKPP case. 
% % The SGIP method still accurately captures the wave front propagation behavior and maintains consistency with FDM results, demonstrating the robustness of the SGIP method when dealing with more complex nonlinear reaction terms.

% Finally, we consider a reaction-diffusion equation with an Arrhenius-type reaction term:
% \begin{equation}
%  u_t = u_{xx} + e^{-E/u}(1-u),
% \end{equation}
% where we choose $E=\frac12$.

\begin{figure}[htbp]
    \centering
    \vspace{-0.2cm}
    \includegraphics[width=0.7\textwidth]{picture/kpp_comparison_exp_all_times.pdf}
\caption{Comparison of FDM and SGIP solutions for $u_t = u_{xx} + e^{-\frac{1}{2u}}(1-u)$}
\label{uxx+exp(1-u)}
\end{figure}

% This reaction term possesses a singularity near $u=0$, posing additional challenges for numerical methods. The SGIP method demonstrates satisfactory performance in capturing the wave propagation under this singular nonlinearity. It is observed in Figure \ref{uxx+exp(1-u)} that the wave front propagates at a notably slower speed compared to the standard FKPP case, while maintaining a coherent wave profile. These results further validate the robustness of the SGIP method when dealing with more complex nonlinear reaction terms.

To quantitatively compare the effects of different reaction terms on propagation speed, we plot the evolution of wave front position over time for all three cases in Figure~\ref{position_diff_reaction}. These numerical experiments comprehensively demonstrate the effectiveness and accuracy of the SGIP method for solving various types of reaction-diffusion equations.

\begin{figure}[htbp]
    \centering
    \vspace{-0.2cm}
    \begin{subfigure}[t]{0.32\textwidth}
        \centering
        \includegraphics[width=0.95\linewidth]{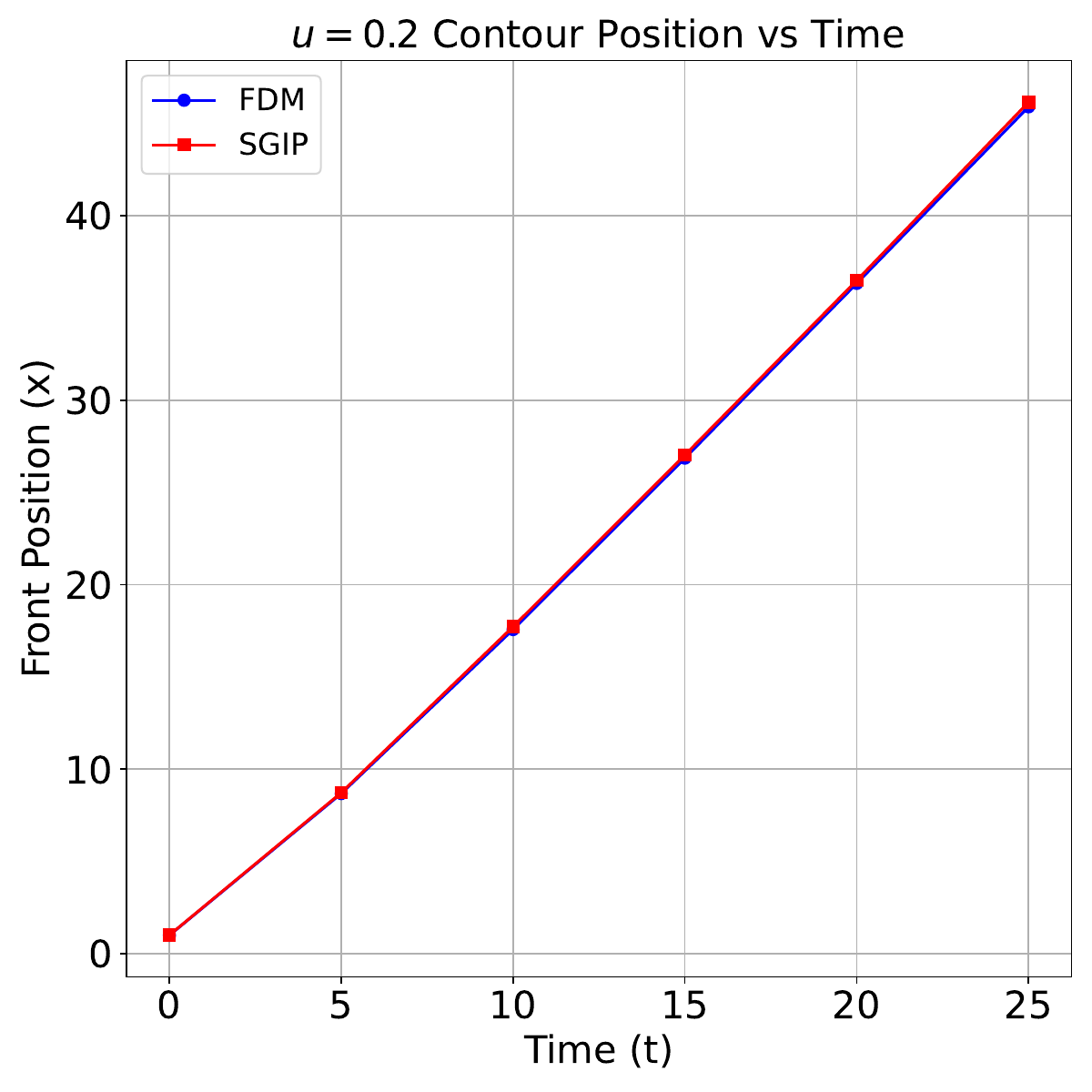}
        \caption{$u_t = u_{xx} + u(1-u)$}
    \end{subfigure}
    \hfill
    \begin{subfigure}[t]{0.32\textwidth}
        \centering
        \includegraphics[width=0.95\linewidth]{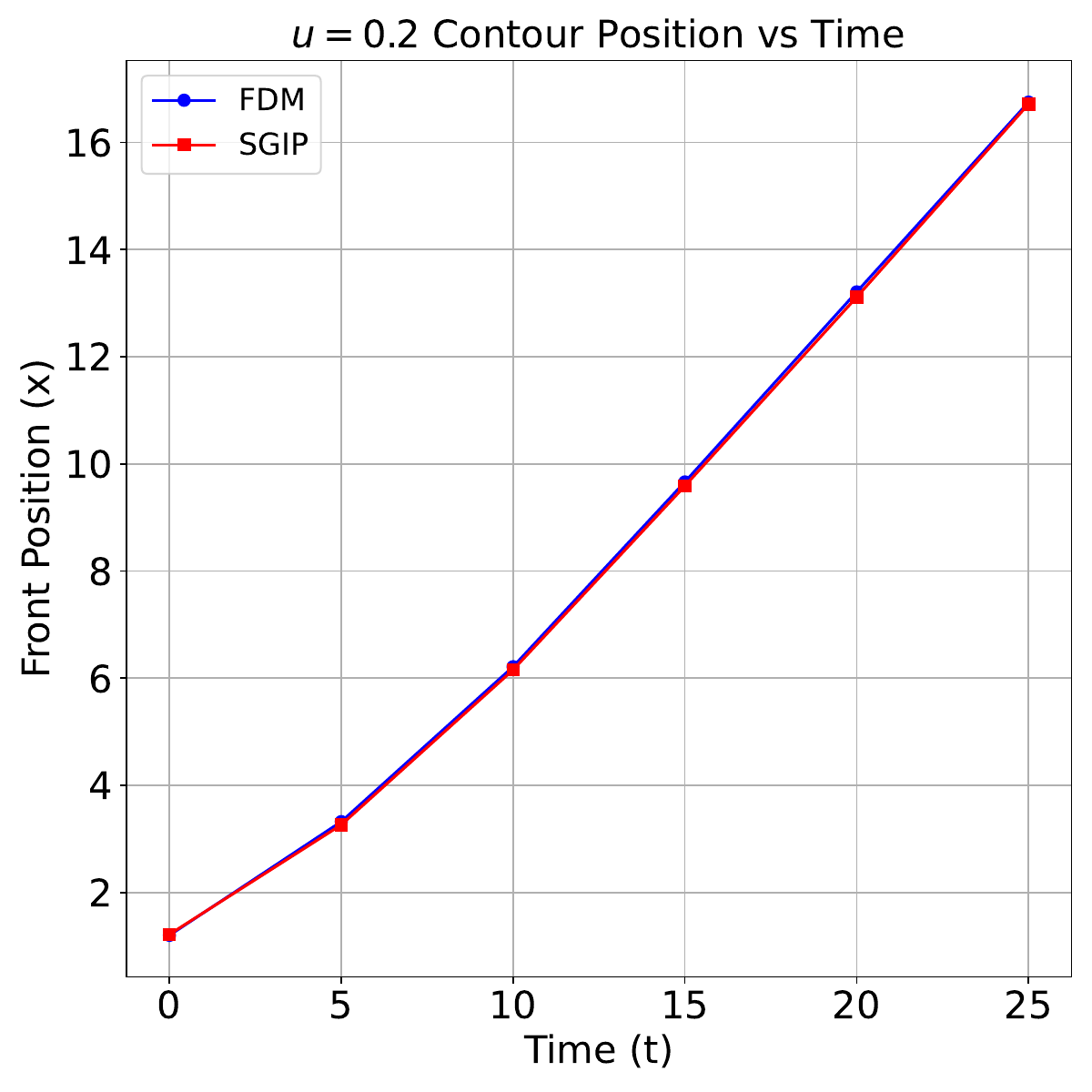}
        \caption{$u_t = u_{xx} + u^2(1-u)$}
    \end{subfigure}
    \hfill
    \begin{subfigure}[t]{0.32\textwidth}
        \centering
        \includegraphics[width=0.95\linewidth]{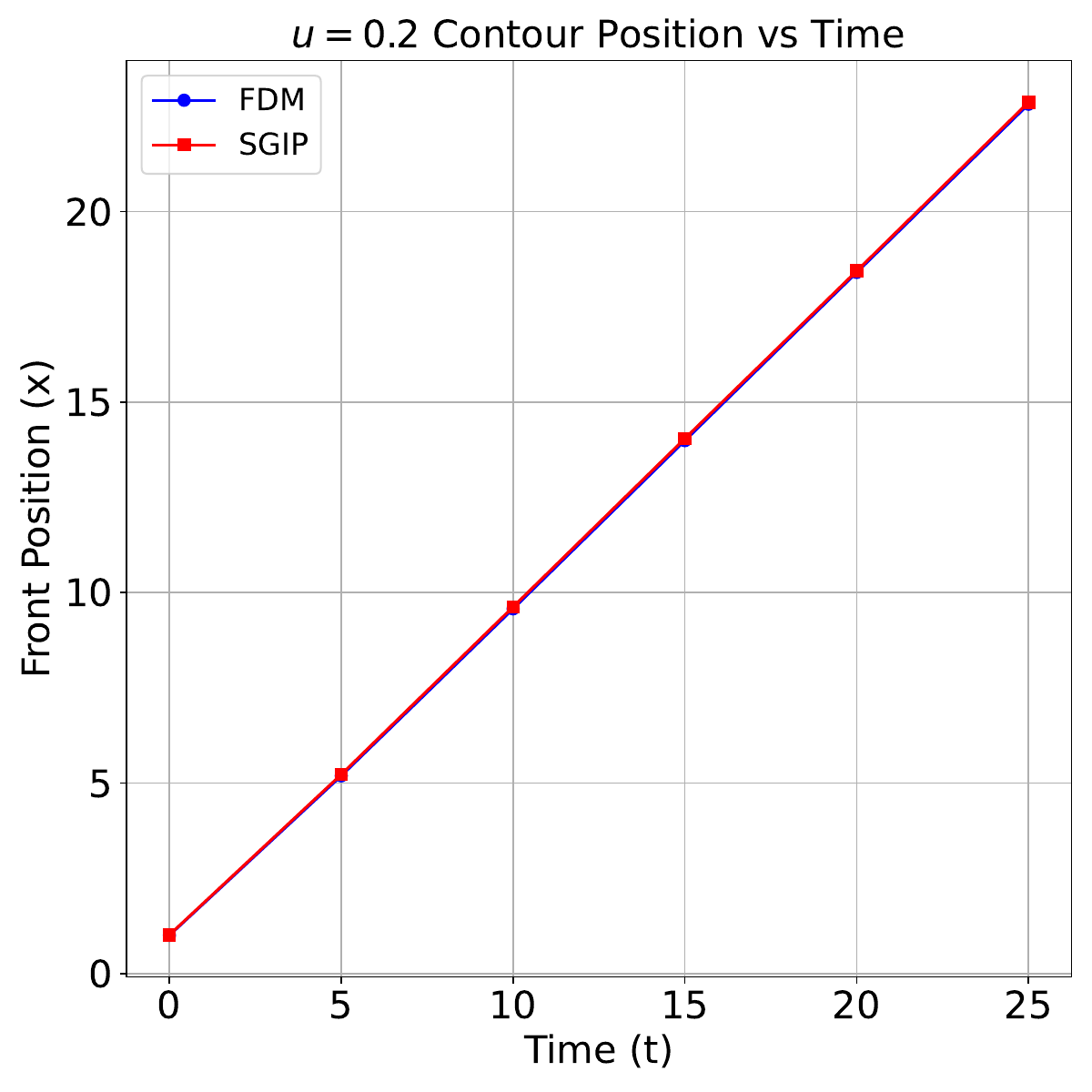}
        \caption{$u_t = u_{xx} + e^{-\frac{1}{2u}}(1-u)$}
    \end{subfigure}
\caption{Propagation speed comparison for particle and FDM}
\label{position_diff_reaction}
\end{figure}

For the bistable reaction-diffusion equation $u_t=u_{xx}+u(1-u)(u-0.25)$, we report results for the two initial conditions $\chi_{[0,1]}$ and $\chi_{[-10,10]}$.
% the latter one can activate the 
\begin{figure}[htbp]
    \centering
    \vspace{-0.2cm}
    \includegraphics[width=0.7\textwidth]{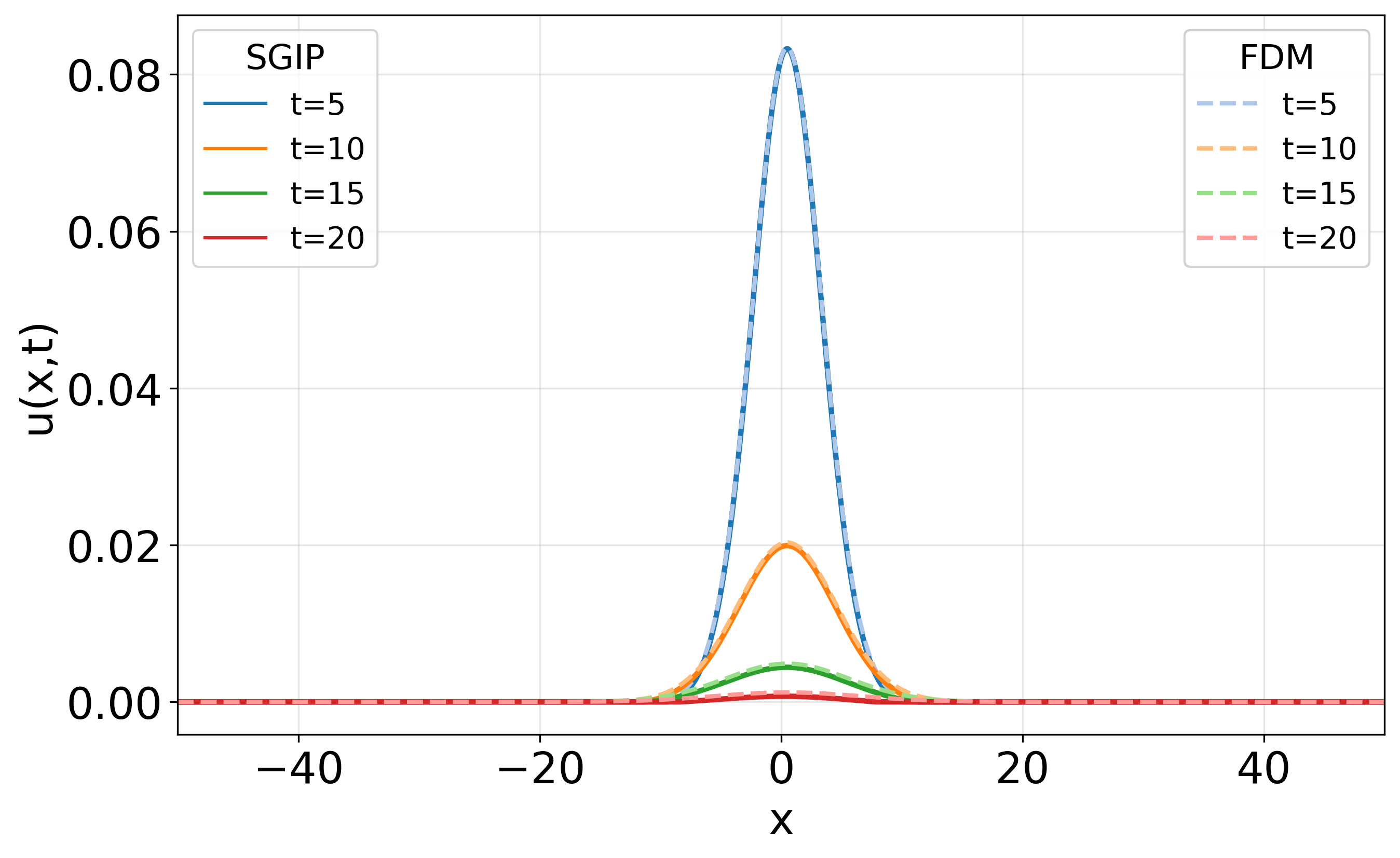}
\caption{Comparison of FDM and SGIP solutions for $u_t = u_{xx} + u(1-u)(u-0.25)$ with initial condition $\chi_{[0,1]}$}
\end{figure}

\begin{figure}[htbp]
    \centering
    \vspace{-0.2cm}
    \includegraphics[width=0.7\textwidth]{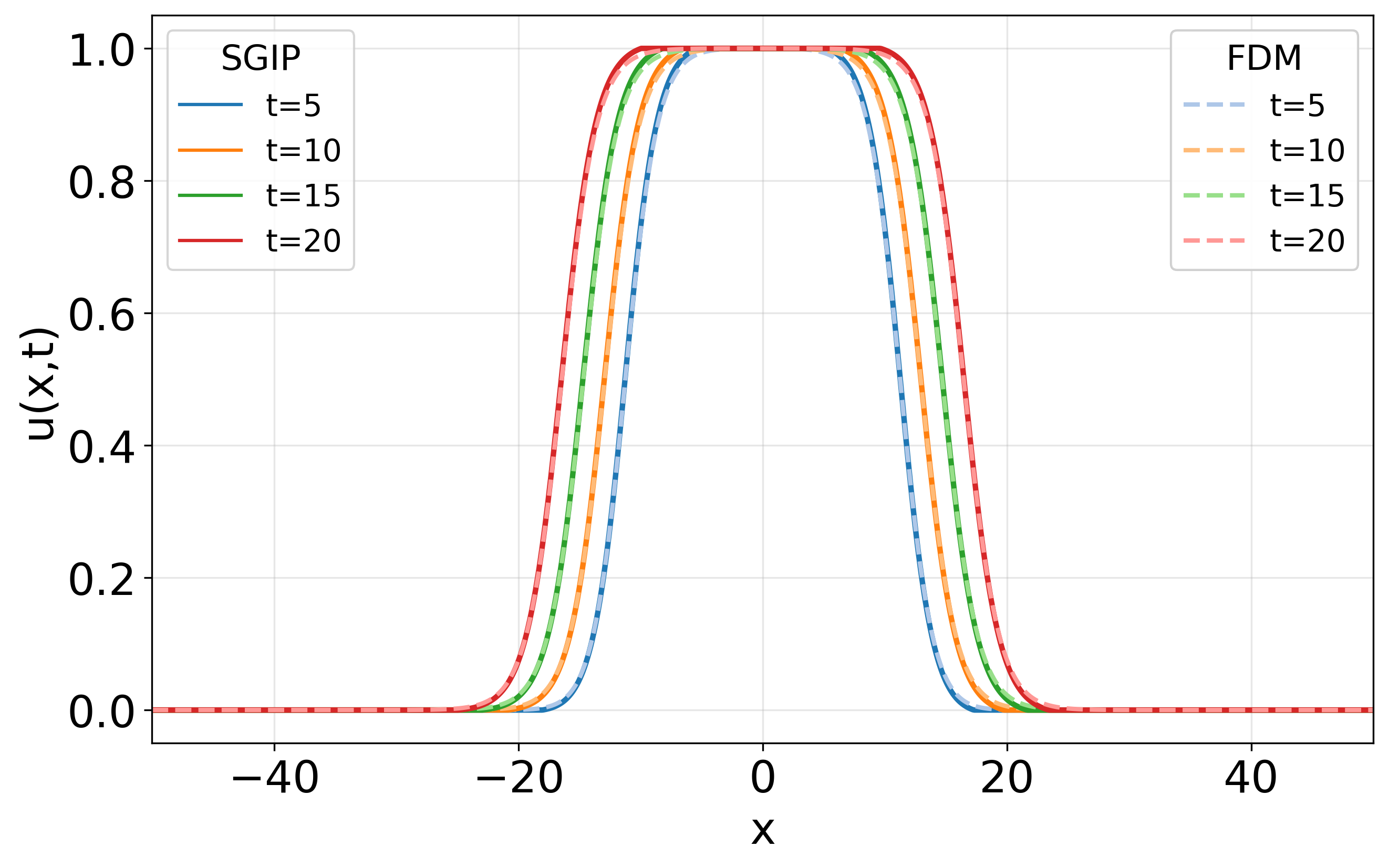}
\caption{Comparison of FDM and SGIP solutions for $u_t = u_{xx} + u(1-u)(u-0.25)$ with initial condition $\chi_{[-10,10]}$}
\end{figure}

\subsection{2D Numerical Experiments}\label{subsection:2d}
We now extend our investigation to 2D RDA systems to further validate the SGIP method's capability in handling more complex spatial dynamics. The 2D experiments provide insights into pattern formation, wave propagation in multiple dimensions, and the interaction between diffusion and advection processes.

% 2D simulations are conducted on the computational domain $\Omega = [-60, 60]^2$. The initial condition is specified by sampling particle positions uniformly from the unit square: $\mathbf{X}_0^i \sim \text{Uniform}([0,1]^2)$ for $i = 1, \dots, N$, corresponding to an initial density $u_0(x,y) = \chi_{[0,1]^2}(x,y)$.

% For the reference FDM simulations, we employ a high-resolution uniform grid with spatial discretization $\delta x = \delta y = 10^{-2}$ and temporal step $\delta t = 10^{-3}$ to ensure accurate reference solutions. 

% For the SGIP method, we use $N = 3 \times 10^6$ particles to approximate the density distribution. The spatial domain is discretized into $10^4$ uniform bins for density estimation, and the time step is set to $\Delta t = 0.5$.

The 2D simulations are conducted on the computational domain $\Omega = [-60, 60]^2$, with the initial condition specified by uniformly sampling particle positions from the unit square: $\mathbf{X}_0^i \sim \text{Uniform}([0,1]^2)$ for $i = 1, \dots, N$, corresponding to an initial density $u_0(x,y) = \chi_{[0,1]^2}(x,y)$. For the reference FDM simulations, we employ a high-resolution uniform grid with spatial discretization $\delta x = \delta y = 10^{-2}$ and temporal step $\delta t = 10^{-3}$ to ensure accurate reference solutions. For the SGIP method, $N = 3 \times 10^6$ particles are used to approximate the density distribution, with the spatial domain discretized into $10^4$ uniform bins for density estimation and a time step of $\Delta t = 0.5$.

\begin{figure}[htbp]
    \centering
    \vspace{-0.2cm}
    \begin{subfigure}[t]{0.48\textwidth}
        \centering
        \includegraphics[width=0.95\linewidth]{picture/KPP_2d_contour_t_5.pdf}
        \caption{t=5}
    \end{subfigure}
    \hfill
    \begin{subfigure}[t]{0.48\textwidth}
        \centering
        \includegraphics[width=0.95\linewidth]{picture/KPP_2d_contour_t_10.pdf}
        \caption{t=10}
    \end{subfigure}
    
    \vspace{0.2cm}
    \begin{subfigure}[t]{0.48\textwidth}
        \centering
        \includegraphics[width=0.95\linewidth]{picture/KPP_2d_contour_t_15.pdf}
        \caption{t=15}
    \end{subfigure}
    \hfill
    \begin{subfigure}[t]{0.48\textwidth}
        \centering
        \includegraphics[width=0.95\linewidth]{picture/KPP_2d_contour_t_20.pdf}
        \caption{t=20}
    \end{subfigure}
   \caption{Comparison of FDM and SGIP solutions for $u_t = u_{xx} + u_{yy} + u(1-u)$}
   \label{uxx+uyy+u(1-u):symmetric}
\end{figure}

% \begin{figure}[htbp]
%     \centering
%     \vspace{-0.2cm}
%     \includegraphics[width=0.6\textwidth]{picture/2d_position.pdf}
%     \caption{Propagation speed comparison for particle and FDM}
%     \label{2d_propagation}
% \end{figure}

Figure~\ref{uxx+uyy+u(1-u):symmetric} demonstrates the evolution of the 2D FKPP equation with the unit-square initial condition, showing contour lines at $u = 0.1$, $0.5$, and $0.9$. The SGIP method successfully preserves the circular symmetry of the expanding wave front, maintaining rotational invariance in the contour morphology across all three concentration levels. 
% Quantitative comparison in Figure \ref{2d_propagation} assesses the radial propagation speeds by tracking the intersection points of the $u=0.2$ contour with the x-axis for both SGIP and FDM. The excellent agreement between methods, evidenced by the linear position-time relationship, confirms the constant speed propagation characteristic of FKPP-type equations and validates the accuracy of our SGIP method in capturing two-dimensional front dynamics.

\subsubsection{2D steady flow}

We now investigate the interaction between reaction-diffusion dynamics and advection terms, which is particularly relevant in biological and environmental applications.

% \textbf{Shear Flow:}
% \begin{equation}
%     u_t = D\Delta u + (\sin y, 0) \cdot \nabla u + u(1-u).
% \end{equation}

We examine the interaction between reaction-diffusion dynamics and three representative steady flows through the model \( u_t = D\Delta u + \mathbf{w} \cdot \nabla u + u(1-u) \). Here $\mathbf{w}$ is the prescribed divergence-free flow, and the particle drift in the notation of Sections~\ref{section_alg}--\ref{section_analysis} is $\mathbf{v}=-\mathbf{w}$. For shear flow \( \mathbf{w} = (\sin y, 0) \), Figure~\ref{shear_flow} demonstrates that the SGIP accurately captures the characteristic shear-aligned patterns, showing how the flow stretches and orients the chemical waves. In cellular flow \( \mathbf{w} = (-\sin x\cos y,\cos x\sin y) \), Figure~\ref{cellular_flow} reveals how the Lagrangian formulation naturally resolves complex recirculation zones, organizing reaction fronts into circular structures within individual cells, while inducing corrugations along intercellular boundaries due to hyperbolic stagnation points.

For the cat's eye flow \( \mathbf{w} = (-\sin x\cos y,\cos x\sin y) + 2(\cos x\sin y, -\sin x\cos y) \), Figure~\ref{cateye_flow} illustrates the competition between diffusion and advection across two regimes: lower diffusivity (\( D=0.5 \)) produces sharp lobe boundaries formed by alternating elliptic islands and hyperbolic channels, while higher diffusivity (\( D=1 \)) enhances cross-channel mixing and smooths front morphology, demonstrating how flow superposition modulates front propagation through competing recirculation and transport regions.

\begin{figure}[htbp]
    \centering
    \vspace{-0.2cm}
    \begin{subfigure}[t]{0.48\textwidth}
        \centering
        \includegraphics[width=0.95\linewidth]{picture/SGIP_shear_t20.pdf}
        \caption{SGIP}
    \end{subfigure}
    \hfill
    \begin{subfigure}[t]{0.48\textwidth}
        \centering
        \includegraphics[width=0.95\linewidth]{picture/FDM_shear_t20.pdf}
        \caption{FDM}
    \end{subfigure}
    \caption{SGIP vs. FDM on shear flow at $t=20$}
    \label{shear_flow}
\end{figure}

% \textbf{Cellular Flow:}
% \begin{equation}
%     u_t=D\Delta u+(-\sin x\cos y,\cos x\sin y)\cdot\nabla u+u(1-u).
% \end{equation}

\begin{figure}[htbp]
    \centering
    \vspace{-0.2cm}
    \begin{subfigure}[t]{0.48\textwidth}
        \centering
        \includegraphics[width=0.95\linewidth]{picture/SGIP_cellular_t20.pdf}
        \caption{SGIP}
    \end{subfigure}
    \hfill
    \begin{subfigure}[t]{0.48\textwidth}
        \centering
        \includegraphics[width=0.95\linewidth]{picture/FDM_cellular_t20.pdf}
        \caption{FDM}
    \end{subfigure}
    \caption{SGIP vs. FDM on cellular flow at $t=20$}
    \label{cellular_flow}
\end{figure}

% \textbf{Cat's Eye Flow:}
% \begin{equation}
%     u_t=D\Delta u+[(-\sin x\cos y,\cos x\sin y)+\delta \cdot(\cos x\sin y, -\sin x\cos y)]\cdot\nabla u+u(1-u).
% \end{equation}
% Here we set $\delta=2$.

\begin{figure}[htbp]
    \centering
    \vspace{-0.2cm}
    \begin{subfigure}[t]{0.48\textwidth}
        \centering
        \includegraphics[width=0.95\linewidth]{picture/SGIP_cateye_t20_D05.pdf}
        \caption{SGIP with $D=0.5$}
    \end{subfigure}
    \hfill
    \begin{subfigure}[t]{0.48\textwidth}
        \centering
        \includegraphics[width=0.95\linewidth]{picture/FDM_cateye_t20_D05.pdf}
        \caption{FDM with $D=0.5$}
    \end{subfigure}
    \vspace{0.2cm}
    \begin{subfigure}[t]{0.48\textwidth}
        \centering
        \includegraphics[width=0.95\linewidth]{picture/SGIP_cateye_t20_D1.pdf}
        \caption{SGIP with $D=1$}
    \end{subfigure}
    \hfill
    \begin{subfigure}[t]{0.48\textwidth}
        \centering
        \includegraphics[width=0.95\linewidth]{picture/FDM_cateye_t20_D1.pdf}
        \caption{FDM with $D=1$}
    \end{subfigure}
    \caption{SGIP vs. FDM on cat's eye flow at $t=20$}
    \label{cateye_flow}
\end{figure}

% Figure \ref{shear_flow} demonstrates the effect of sinusoidal shear flow on reaction-diffusion dynamics. The SGIP accurately captures the formation of characteristic shear-aligned patterns, showing how the flow stretches and orients chemical waves along the shear direction. Figure \ref{cellular_flow} shows the interaction with cellular flow patterns, where the Lagrangian formulation naturally resolves complex recirculation zones and their impact on front morphology. The flow organizes reaction fronts into circular structures confined within individual cells, while inducing corrugations along intercellular boundaries due to hyperbolic stagnation points. Figure \ref{cateye_flow} examines a more intricate cat's-eye flow combining cellular and shear components. The SGIP reproduces the competition between diffusion and advection across different regimes ($D=0.5$ and $D=1$), capturing characteristic lobe structures formed by alternating elliptic islands and hyperbolic channels. Lower diffusivity ($D=0.5$) produces sharp lobe boundaries, while higher diffusivity ($D=1$) enhances cross-channel mixing and smooths front morphology, illustrating how flow superposition modulates front propagation through competing recirculation and transport regions.

\subsection{3D Numerical Experiments}\label{subsection:3d}

%\subsubsection{ABC flow}
We now extend our investigation to 3D RDA systems, applying our methods to the ABC flow, which is described by
\begin{equation}
    u_t=D \Delta u+(A\sin z + C\cos y, B\sin x + A\cos z, C\sin y + B\cos x)\cdot\nabla u+u(1-u),
\end{equation}
where $A=1, B=\sqrt{\frac{2}{3}}, C=\sqrt{\frac{1}{3}}$. Simulations are conducted on the domain $[-60,60]^3$ with initial condition $u_0(\mathbf{x}) = \chi_{B(0,1)}(\mathbf{x})$, where $\chi_{B(0,1)}$ denotes the characteristic function of the unit ball centered at the origin.

% The SGIP employs $5\times 10^6$ particles with $100^3$ spatial bins and time step $\Delta t = 0.5$. In contrast, the FDM requires significantly finer discretization for numerical stability: time step $\delta t = 10^{-3}$ and spatial grid of at least $300^3$ points, resulting in substantially higher computational cost.

\begin{figure}[htbp]
    \centering
    \vspace{-0.2cm}
    % \begin{subfigure}[t]{0.48\textwidth}
    %     \centering
    %     \includegraphics[width=0.95\linewidth]{picture/particle_method_3d_ABC_t20_x0_KPP_D1.pdf}
    %     \caption{SGIP with projection at x=0 plane}
    % \end{subfigure}
    % \hfill
    % \begin{subfigure}[t]{0.48\textwidth}
    %     \centering
    %     \includegraphics[width=0.95\linewidth]{picture/fdm_3d_ABC_t20_x0_KPP_D1.pdf}
    %     \caption{FDM with projection at x=0 plane}
    % \end{subfigure}
    % \vspace{0.2cm}
    \begin{subfigure}[t]{0.48\textwidth}
        \centering
        \includegraphics[width=0.95\linewidth]{picture/particle_method_3d_ABC_t20_y0_KPP_D1.pdf}
        \caption{SGIP with projection at $y=0$ plane}
    \end{subfigure}
    \hfill
    \begin{subfigure}[t]{0.48\textwidth}
        \centering
        \includegraphics[width=0.95\linewidth]{picture/fdm_3d_ABC_t20_y0_KPP_D1.pdf}
        \caption{FDM with projection at $y=0$ plane}
    \end{subfigure}
    % \vspace{0.2cm}
    % \begin{subfigure}[t]{0.48\textwidth}
    %     \centering
    %     \includegraphics[width=0.95\linewidth]{picture/particle_method_3d_ABC_t20_z0_KPP_D1.pdf}
    %     \caption{SGIP with projection at z=0 plane}
    % \end{subfigure}
    % \hfill
    % \begin{subfigure}[t]{0.48\textwidth}
    %     \centering
    %     \includegraphics[width=0.95\linewidth]{picture/fdm_3d_ABC_t20_z0_KPP_D1.pdf}
    %     \caption{FDM with projection at z=0 plane}
    % \end{subfigure}
    \caption{SGIP vs. FDM on ABC flow with $D=1$ at $t=20$}
    \label{ABCflow_D=1_u(1-u)}
\end{figure}

\begin{figure}[htbp]
    \centering
    \vspace{-0.2cm}
    % \begin{subfigure}[t]{0.48\textwidth}
    %     \centering
    %     \includegraphics[width=0.95\linewidth]{picture/particle_method_3d_ABC_t20_x0_KPP_D05.pdf}
    %     \caption{SGIP with projection at x=0 plane}
    % \end{subfigure}
    % \hfill
    % \begin{subfigure}[t]{0.48\textwidth}
    %     \centering
    %     \includegraphics[width=0.95\linewidth]{picture/fdm_3d_ABC_t20_x0_KPP_D05.pdf}
    %     \caption{FDM with projection at x=0 plane}
    % \end{subfigure}
    % \vspace{0.2cm}
    \begin{subfigure}[t]{0.48\textwidth}
        \centering
        \includegraphics[width=0.95\linewidth]{picture/particle_method_3d_ABC_t20_y0_KPP_D05.pdf}
        \caption{SGIP with projection at y=0 plane}
    \end{subfigure}
    \hfill
    \begin{subfigure}[t]{0.48\textwidth}
        \centering
        \includegraphics[width=0.95\linewidth]{picture/fdm_3d_ABC_t20_y0_KPP_D05.pdf}
        \caption{FDM with projection at y=0 plane}
    \end{subfigure}
    % \vspace{0.2cm}
    % \begin{subfigure}[t]{0.48\textwidth}
    %     \centering
    %     \includegraphics[width=0.95\linewidth]{picture/particle_method_3d_ABC_t20_z0_KPP_D05.pdf}
    %     \caption{SGIP with projection at z=0 plane}
    % \end{subfigure}
    % \hfill
    % \begin{subfigure}[t]{0.48\textwidth}
    %     \centering
    %     \includegraphics[width=0.95\linewidth]{picture/fdm_3d_ABC_t20_z0_KPP_D05.pdf}
    %     \caption{FDM with projection at z=0 plane}
    % \end{subfigure}
    \caption{SGIP vs. FDM on ABC flow with $D=0.5$ at $t=20$}
    \label{ABCflow_D=0.5_u(1-u)}
\end{figure}

\begin{figure}[htbp]
    \centering
    \vspace{-0.2cm}
    \begin{subfigure}[t]{0.32\textwidth}
        \centering
        \includegraphics[width=0.95\linewidth]{picture/particle_method_3d_ABC_t20_x0_KPP_D01.pdf}
        \caption{SGIP with projection at $x=0$ plane}
    \end{subfigure}
    \begin{subfigure}[t]{0.32\textwidth}
        \centering
        \includegraphics[width=0.95\linewidth]{picture/particle_method_3d_ABC_t20_y0_KPP_D01.pdf}
        \caption{SGIP with projection at $y=0$ plane}
    \end{subfigure}
    \begin{subfigure}[t]{0.32\textwidth}
        \centering
        \includegraphics[width=0.95\linewidth]{picture/particle_method_3d_ABC_t20_z0_KPP_D01.pdf}
        \caption{SGIP with projection at $z=0$ plane}
    \end{subfigure}
    \caption{SGIP on ABC flow with $D=0.1$}
    \label{ABCflow_D=0.1_u(1-u)}
\end{figure}

The SGIP employs $5\times 10^6$ particles with $100^3$ spatial bins and a time step $\Delta t = 0.5$. Figures~\ref{ABCflow_D=1_u(1-u)} and \ref{ABCflow_D=0.5_u(1-u)} compare both methods on cross-sections through the coordinate planes for diffusion coefficients $D=1$ and $D=0.5$, showing excellent agreement. Since the cross-sectional structures at $x=0$ and $z=0$ are qualitatively similar to those at $y=0$, only the $y=0$ plane is displayed for clarity.

A key advantage of SGIP is its low and stable computational cost. For $D=1$, SGIP requires only 76 s, while the FDM (with a $200^3$ grid and $\delta t = 5\times10^{-3}$) takes about 337 s. As $D$ decreases to 0.5 (Figure~\ref{ABCflow_D=0.5_u(1-u)}), advection and reaction begin to dominate. The SGIP method, using the same parameters ($5 \times 10^6$ particles, $100^3$ bins, $\Delta t = 0.5$), maintained a similar computation time, demonstrating its robustness without requiring parameter refinement. In contrast, the FDM required a significantly finer discretization (a $300^3$ grid, $\delta t = 10^{-4}$) to avoid numerical instability, which increased the runtime to approximately $9\times10^{3}$ s.

\enlargethispage{\baselineskip}
As $D$ decreases further to 0.1 (Figure~\ref{ABCflow_D=0.1_u(1-u)}), the FDM demands an even finer resolution ($400^3$ grid points) that becomes computationally prohibitive due to excessive memory requirements on conventional workstations. The SGIP method, however, remains feasible and accurately captures the complex 3D flow structures across all diffusion regimes. \by{The computational advantage is most pronounced in the advective regimes tested here, where FDM must resolve strongly deformed fronts. Without advection, this resolution demand is reduced, while SGIP still incurs density reconstruction and resampling costs, so the relative gain can be smaller.}

% Figures \ref{ABCflow_D=1_u(1-u)} and \ref{ABCflow_D=0.5_u(1-u)} compare both methods on cross-sections through the coordinate planes for diffusion coefficients $D=1$ and $D=0.5$, showing excellent agreement. As $D$ decreases to 0.1 (Figure \ref{ABCflow_D=0.1_u(1-u)}), advection and reaction dominate the dynamics, requiring even finer FDM discretization ($400^3$ grid points) that becomes computationally prohibitive due to excessive memory requirements on conventional workstations. The SGIP method, however, remains feasible and accurately captures the complex 3D flow structures across all diffusion regimes.

We further validate the SGIP using more complex nonlinear reaction terms (cubic, Arrhenius), maintaining the same computational parameters. The unit ball initial condition is extended to $u_0(\mathbf{x}) = \chi_{B(0,5)}(\mathbf{x})$ to accommodate the distinct reaction kinetics and ensure clear visualization of the ensuing dynamics. 
%However, since these new reaction terms result in markedly different reaction kinetics, the previously used unit ball initial condition would produce dynamics that are too confined and difficult to visualize effectively within the computational domain. To compensate for this and to generate clearly observable front propagation or pattern formation, the initial condition is modified to $u_0(\mathbf{x}) = \chi_{B(0,5)}(\mathbf{x})$ on a larger support. 

\begin{figure}[htbp]
    \centering
    \vspace{-0.2cm}
    % \begin{subfigure}[t]{0.48\textwidth}
    %     \centering
    %     \includegraphics[width=0.95\linewidth]{picture/particle_method_3d_ABC_t20_x0_nonlinear.pdf}
    %     \caption{SGIP with projection at x=0 plane}
    % \end{subfigure}
    % \hfill
    % \begin{subfigure}[t]{0.48\textwidth}
    %     \centering
    %     \includegraphics[width=0.95\linewidth]{picture/fdm_3d_ABC_t20_x0_nonlinear.pdf}
    %     \caption{FDM with projection at x=0 plane}
    % \end{subfigure}
    % \vspace{0.2cm}
    \begin{subfigure}[t]{0.48\textwidth}
        \centering
        \includegraphics[width=0.95\linewidth]{picture/particle_method_3d_ABC_t20_y0_nonlinear.pdf}
        \caption{SGIP with projection at $y=0$ plane}
    \end{subfigure}
    \hfill
    \begin{subfigure}[t]{0.48\textwidth}
        \centering
        \includegraphics[width=0.95\linewidth]{picture/fdm_3d_ABC_t20_y0_nonlinear.pdf}
        \caption{FDM with projection at $y=0$ plane}
    \end{subfigure}
    % \vspace{0.2cm}
    % \begin{subfigure}[t]{0.48\textwidth}
    %     \centering
    %     \includegraphics[width=0.95\linewidth]{picture/particle_method_3d_ABC_t20_z0_nonlinear.pdf}
    %     \caption{SGIP with projection at z=0 plane}
    % \end{subfigure}
    % \hfill
    % \begin{subfigure}[t]{0.48\textwidth}
    %     \centering
    %     \includegraphics[width=0.95\linewidth]{picture/fdm_3d_ABC_t20_z0_nonlinear.pdf}
    %     \caption{FDM with projection at z=0 plane}
    % \end{subfigure}
    \caption{SGIP vs. FDM on ABC Flow with Reaction Term: $u^2(1-u)$ at $t=20$}
    \label{ABCflow_u^2(1-u)}
\end{figure}

\begin{figure}[htbp]
    \centering
    \vspace{-0.2cm}
    % \begin{subfigure}[t]{0.48\textwidth}
    %     \centering
    %     \includegraphics[width=0.95\linewidth]{picture/particle_method_3d_ABC_t20_x0_exp.pdf}
    %     \caption{SGIP with projection at x=0 plane}
    % \end{subfigure}
    % \hfill
    % \begin{subfigure}[t]{0.48\textwidth}
    %     \centering
    %     \includegraphics[width=0.95\linewidth]{picture/fdm_3d_ABC_t20_x0_exp.pdf}
    %     \caption{FDM with projection at x=0 plane}
    % \end{subfigure}
    % \vspace{0.2cm}
    \begin{subfigure}[t]{0.48\textwidth}
        \centering
        \includegraphics[width=0.95\linewidth]{picture/particle_method_3d_ABC_t20_y0_exp.pdf}
        \caption{SGIP with projection at $y=0$ plane}
    \end{subfigure}
    \hfill
    \begin{subfigure}[t]{0.48\textwidth}
        \centering
        \includegraphics[width=0.95\linewidth]{picture/fdm_3d_ABC_t20_y0_exp.pdf}
        \caption{FDM with projection at $y=0$ plane}
    \end{subfigure}
    % \vspace{0.2cm}
    % \begin{subfigure}[t]{0.48\textwidth}
    %     \centering
    %     \includegraphics[width=0.95\linewidth]{picture/particle_method_3d_ABC_t20_z0_exp.pdf}
    %     \caption{SGIP with projection at z=0 plane}
    % \end{subfigure}
    % \hfill
    % \begin{subfigure}[t]{0.48\textwidth}
    %     \centering
    %     \includegraphics[width=0.95\linewidth]{picture/fdm_3d_ABC_t20_z0_exp.pdf}
    %     \caption{FDM with projection at z=0 plane}
    % \end{subfigure}
    \caption{SGIP vs. FDM on ABC Flow with Reaction Term: $e^{-\frac{1}{2u}}(1-u)$ at $t=20$}
    \label{ABCflow_exp(1-u)}
\end{figure}

Figure~\ref{ABCflow_u^2(1-u)} demonstrates the method's performance with cubic reaction term $u^2(1-u)$, while Figure~\ref{ABCflow_exp(1-u)} shows results for the Arrhenius-type reaction $e^{-\frac{1}{2u}}(1-u)$. In both cases, the SGIP accurately captures the complex interplay between advection, diffusion, and nonlinear reaction dynamics, showing excellent agreement with FDM reference solutions.

The consistent performance across different reaction types and initial conditions demonstrates the robustness and versatility of SGIP. Its Lagrangian formulation treats transport in complex flow fields without requiring an adaptive mesh to track the deformed front. The ability to maintain accuracy with relatively large time steps ($\Delta t = 0.5$ versus FDM's $\delta t = 10^{-3}$) and comparatively coarse reconstruction bins provides a significant computational advantage for the 3D RDA problems considered here.

\section{Conclusion}\label{section_con}

We have presented the SGIP method, a novel stochastic genetic interacting particle algorithm for solving reaction-diffusion-advection equations. The method's effectiveness stems from an integrated design that combines operator splitting for process decoupling, stochastic particles for transport, and adaptive resampling for long-term stability.

The main theoretical contribution is a complete convergence analysis that quantifies the dependence of the error on the time step, bin size, particle number, and spatial dimension. Numerical experiments consider different reaction terms (FKPP, cubic, and Arrhenius) and complex flows (shear, cellular, and ABC) over 1D to 3D domains. In the 3D examples, the Lagrangian transport step reduces the fine-grid resolution required by Eulerian methods while maintaining stable computations with relatively large time steps.

The adaptive resampling strategy is crucial for sustained stability and efficiency. By intelligently redistributing particles based on Eulerian density reconstruction, it balances statistical accuracy with computational cost. Future work will extend SGIP to multi-species systems and implement the fast adaptive algorithm discussed in Remark~\ref{remark_alg} to concentrate computational resources in active transition regions, thereby improving efficiency for higher-dimensional problems. \by{SGIP may also be extended to solution-dependent or dynamically coupled velocity fields; numerical validation and convergence analysis of such extensions are left for future work.} \by{Another practically important direction is to incorporate physical boundary conditions and domain-adapted density reconstruction for applications in confined or irregular geometries.}

\section*{Acknowledgements}
ZZ was partly supported by the National Natural Science Foundation of China (Projects 92470103 and 12171406), the Hong Kong RGC grant (Projects 17304324 and 17300325), the Seed Funding Programme for Basic Research (HKU), and the Hong Kong RGC Research Fellow Scheme 2025. ZW was partly supported by NTU SUG-023162-00001, MOE AcRF Tier 1 Grant RG17/24 and MOE AcRF Tier 2 Grant MOE-T2EP20225-0007. JX was partly supported by NSF grants DMS-2219904 and DMS-2309520, a Qualcomm Gift Award; by the Swedish Research Council grant no. 2021-06594 during his residence at the Institut Mittag-Leffler in Djursholm, Sweden, and a funding for the Thematic Programme 
``Free Boundary Problems'' 
during his stay at the 
 E. Schr\"odinger Institute, Vienna, Austria, both in the Fall of 2025.  The computations were performed using the HKU ITS research computing facilities and 
the Greenplanet Cluster at UC Irvine.

\by{\section*{Data and Code Availability}
Code for reproducing the numerical experiments will be made publicly available at
\url{https://github.com/hubyyyy/SGIP_code}. No external input data were used.}

\bibliographystyle{siamplain}
\bibliography{reference}
\end{document}